\documentclass{article}

\usepackage{amsmath,amsfonts,bm, amsthm}

\def\eqref#1{equation~\ref{#1}}

\def\1{\bm{1}}

\DeclareMathAlphabet{\mathsfit}{\encodingdefault}{\sfdefault}{m}{sl}
\SetMathAlphabet{\mathsfit}{bold}{\encodingdefault}{\sfdefault}{bx}{n}

\usepackage{url}
\usepackage{amsthm}
\usepackage{amsmath}
\usepackage{amssymb}
\usepackage[colorlinks,linkcolor=violet, citecolor = violet]{hyperref}
\usepackage[nameinlink]{cleveref}
\usepackage{nameref}
\usepackage{mywidebar}
\usepackage{amsfonts}       
\usepackage{tikz-cd}
\usepackage{xcolor} 
\usepackage{graphicx}  
\usepackage{subcaption}  
\usepackage{multirow}
\usepackage{booktabs}
\usepackage{mathrsfs} 
\usepackage{natbib}
\usepackage{authblk}
\usepackage{dsfont}
\usepackage{verbatim}
\usepackage{graphicx} 
\usepackage{listings}
\usepackage{standalone}

\newcommand{\rank}{\operatorname{rank}}

\newcommand{\Sym}{\operatorname{Sym}}
\newcommand{\diag}{\operatorname{diag}}
\newcommand{\Res}{\operatorname{Res}}
\newcommand{\PP}{\mathbb{P}}

\newcommand{\CC}{\mathbb{C}}

\theoremstyle{definition}
\newtheorem{theorem}{Theorem}
\newtheorem{definition}[theorem]{Definition}
\newtheorem{proposition}[theorem]{Proposition}
\newtheorem{lemma}[theorem]{Lemma}
\newtheorem{corollary}[theorem]{Corollary}

\theoremstyle{definition}
\newtheorem{remark}[theorem]{Remark}
\newtheorem{example}[theorem]{Example}

\newcommand{\exampleqed}{\hfill $\triangle$}
\AtEndEnvironment{example}{\exampleqed}

\usepackage[margin=1in]{geometry}

\title{Algebraic Invariants of Lightning Self-Attention}
\author[1,4]{Yulia Alexandr}
\author[2]{Hao Duan}
\author[1,2,3]{Guido Montúfar}
\affil[1]{Department of Mathematics, UCLA}
\affil[2]{Department of Statistics \& Data Science, UCLA}
\affil[3]{Max Planck Institute for Mathematics in the Sciences, Leipzig}
\affil[4]{School of Mathematics and Statistics, University of Melbourne}
\date{}

\begin{document}

\title{Algebraic Invariants of Lightning Self-Attention}

\maketitle

\begin{abstract}
We study the polynomial coefficients of lightning self-attention as coordinates of an algebraic variety. In the single-token case, we characterize the coefficient variety as a rank-constrained Chow-type variety and derive algebraic equations from this description. For multiple tokens, linear relations reduce the nonlinear geometry to coefficients involving interactions between distinct tokens. We characterize the resulting variety by a common linear factor together with a low-rank condition on the remaining factors. This structure yields explicit families of determinantal, Veronese-type, and Sylvester resultant-based invariants. In the rank-one 
case,  natural pencil and flattening equations define the variety set-theoretically. We complement these results with computations in small dimensions, where the theoretical constructions recover the generators of the defining ideals. Overall, the paper gives a structural description of the algebraic constraints imposed on polynomial coefficient arrays by lightning self-attention and provides explicit certificates of non-realizability. 
\end{abstract}

\vspace{1em}
\noindent \textbf{Keywords}: lightning self-attention, algebraic invariants, implicitization, transformer architecture, determinantal constraints, Chow varieties, Veronese constraints, Sylvester resultants

\vspace{0.5em}
\noindent \textbf{2020 Mathematics Subject Classification}: 13P25, 68T07, 14M12, 15A69

\section{Introduction}

Self-attention is a central component of transformer architectures and, more broadly, modern sequence models. 
It maps a sequence of input vectors, called tokens, to a new sequence by combining information across positions, with weights determined by pairwise interactions between the input tokens and normalized through a softmax.
In this paper, we study the \emph{lightning self-attention} module, also called linear self-attention, obtained by replacing the softmax nonlinearity with the identity map. 
In this setting, the module becomes a cubic polynomial map in the input coordinates, whose coefficients are themselves cubic polynomials in the parameters, making it amenable to analysis with algebraic~methods. 

This polynomial structure leads to a natural problem: which coefficient arrays can be realized by a lightning self-attention module? 
Equivalently, what algebraic relations, within and across output coordinates, must be satisfied by the polynomial coefficients of the attention module? 
From the perspective of algebraic statistics and applied algebraic geometry, this is an implicitization problem for a structured parametrization. 
Algebraic invariants 
encode polynomial relations satisfied by the model and certify non-realizability. 
From the machine learning perspective, the same problem asks what structural constraints are built into the attention mechanism. 
Studying the corresponding algebraic invariants makes these constraints explicit, clarifies the expressive limits of linear self-attention, and provides potential tools for model comparison and verification. 

We treat the polynomial coefficients of lightning self-attention as coordinates of a parametrized algebraic variety, which we call the \textit{attention variety}. 
Our goal is to identify and organize the algebraic invariants satisfied by these coefficients. 
We show that this coefficient geometry is governed by linear relations together with structural factorization and low-rank constraints, which give rise to Chow-type, determinantal, Veronese-type, and Sylvester resultant-type relations. 
To illustrate the nature of these relations, \Cref{fig:quartic} shows a two-dimensional slice of one of the Veronese-type invariants arising in our analysis. 
Even in this small example, the geometry exhibits nontrivial higher-order dependencies among the coefficients, in this case of degree four. 

\begin{figure}[!ht]
    \centering
   \includegraphics[width=0.35\linewidth]{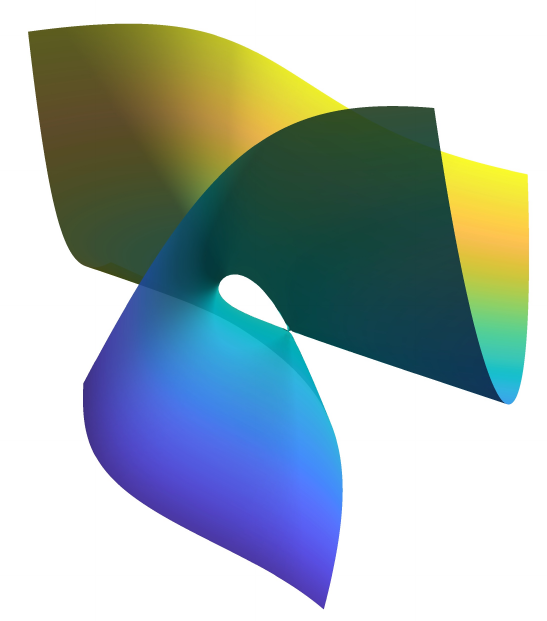}
    \caption{A two-dimensional slice of the Veronese-type quartic of the self-attention module in \Cref{ex:quartic}. 
    } 
    \label{fig:quartic}
\end{figure}

\vspace{-1em}
\subsection*{Main contributions}

This paper identifies mechanisms that force algebraic constraints on the coefficient space of lightning self-attention. 
The invariants 
arise from 
structural properties of the parametrization, especially factorization and symmetrization. 
Factorization appears in several forms: 
in the output polynomials themselves, 
in coefficient tensors and slices, 
and in auxiliary forms built from structured products and contractions of the attention and value parameters. 
In low-rank regimes, 
these factorizations also produce determinantal constraints. 
Symmetrization arises because polynomial coefficients forget the ordering of repeated input variables, while the attention map distinguishes target and context positions. 
%
These mechanisms lead to the linear, Chow-type, low-rank, Veronese-type, and resultant invariants studied below. 
We write $t$ for the sequence length, $d$ for the token dimension, and $a$ for the rank of the attention matrix.
The main contributions are as follows. 

\begin{itemize}

\item We formulate the coefficient characterization problem for lightning self-attention as an implicitization problem. Using scaled polynomial coefficients, we introduce the associated parametrization and its Zariski closure, which we call the \textit{attention variety}. 

\item We identify the linear structure of this variety. For each fixed output row, we prove sequence-copy relations across sequence positions and, 
for $t\geq2$, internal symmetrization relations within a fixed output coordinate, showing that every single-column coefficient is a linear combination of cross-column~coefficients.

\item We characterize the single-column and cross-column coefficient varieties. In the single-token case, the single-column cubics are exactly the products of a linear form and a quadratic form of rank at most $2a$. 
For $t\geq2$, a representative cross-column block is exactly a tuple of quadrics sharing a common linear factor, whose remaining linear factors span a space of dimension at most $a$. After imposing the linear relations, the attention variety is linearly isomorphic to this cross-column variety.

\item We derive families of nonlinear algebraic invariants from these structural descriptions. In the single-token case, we obtain Chow-type relations of split type $(2,1)$, together with additional rank constraints in the low-rank regime. For the cross-column coefficients, we obtain pencil and flattening constraints, determinantal syzygies, Veronese-type constraints, and quartic invariants arising from Sylvester resultants. In the rank-one cross-column case, the pencil and flattening equations cut out the cross-column variety set-theoretically.

\item We complement the theoretical results with reproducible \texttt{Macaulay2} computations for the attention coefficient parametrization. The code constructs the proposed invariants directly from the coefficient formulas, without first computing the full implicit ideal. In all small-dimensional cases where the full defining ideal can be computed for comparison, its minimal generators are explained by the theoretical constructions in this paper. The code and computations supporting the examples are available at \url{https://github.com/yuliaalexandr/algebraic-invariants-of-lightning-self-attention}. 

\end{itemize}

These results give a coefficient-space, or image-space, view of lightning self-attention. 
They reveal an interpretable algebraic structure in the coefficient geometry 
where factorization, low rank, and symmetrization constrain the model image inside the ambient polynomial space and create dependencies across coefficient slices. 
%
From the machine learning perspective, this makes the expressive constraints of the model explicit, giving certificates of non-realizability and suggesting tools for optimization in coefficient space, model comparison and exact verification. 


\subsection*{Related work}

Since the introduction of the transformer~\cite{vaswani2017attention}, attention mechanisms have become a central component of modern machine learning. A growing theoretical literature studies attention from a range of mathematical perspectives, including measure-theoretic formulations~\cite{vuckovic2020mathematical}, Lipschitz continuity and regularity~\cite{kim2021lipschitz}, interacting-particle and mean-field descriptions~\cite{geshkovski2025mathematical}, and spline-theoretic interpretations~\cite{lai2024attention}. 
Other recent works investigate the expressive power and limitations of self-attention and transformer architectures, including universality and approximation~\cite{Yun2020Are}, limitations of relative-position-based attention~\cite{NEURIPS2022_1ba5f641}, structural constraints on self-attention matrices~\cite{10.1609/aaai.v37i7.26055}, and learning-theoretic properties of a single attention layer~\cite{fu2023what}. 
Despite this progress, fundamental questions about optimization and the geometry of the functions realized by attention architectures remain open. 
Closest to the present work is the recent algebro-geometric study of lightning self-attention~\cite{henry2025geometry}, which investigates the parametrization from the point of view of identifiability, image dimension, and singularities. Our perspective is complementary: rather than studying fibers and dimension of the parametrization, we study the image in coefficient space and derive algebraic invariants satisfied by the realized polynomial maps. 

Algebraic invariants have long played an important role in algebraic statistics, particularly graphical models and phylogenetics, where polynomial relations encode structural constraints on the corresponding family of distributions or hypotheses~\cite{allman2003phylogenetic, allman2008phylogenetic, cavender1987invariants, cueto2010geometry, drton2007algebraic, GARCIA2005331, lake1987phylogenetic, seigal2017mixtures}. 
In machine learning, a growing body of work applies algebraic, geometric, and tropical methods to neural networks. 
This literature has addressed questions of expressivity \cite{alexandr2025constraining,finkel2025activation,kileel2019expressive,kubjas2024geometry,pmlr-v80-zhang18i}, 
identifiability 
\cite{allman2009identifiability,henry2025geometry,doi:10.1137/16M1077489,usevich2025identifiability}, 
and optimization 
\cite{brandenburg2024tropical,kohn2022geometry,kohn2024,Trager2019PureAS,venturi2019spurious}. 
This direction is sometimes referred to as \textit{neuroalgebraic geometry}; see~\cite{marchetti2025algebra} for a recent perspective.
Within this area, explicit algebraic invariants and defining equations for neural models have begun to be studied in their own right~\cite{alexandr2025constraining,grosdos2025algebraicgeometryrationalneural,shahverdi2025algebraic}. 
Recent work on ReLU networks~\cite{alexandr2025constraining}, for example, shows that the piecewise-multilinear dependence of the network outputs on the parameters imposes nontrivial polynomial constraints on the set of realizable outputs. 
The present paper develops this perspective for lightning~self-attention.

Beyond their intrinsic geometric interest, algebraic invariants can serve as certificates of non-realizability: any candidate coefficient array or output pattern that violates an invariant cannot be realized by the model. 
They also characterize the feasible set of model behaviors through structured algebraic constraints. 
This perspective has connections to formal verification~\cite{albarghouthi2021intro}, where the goal is to establish provable guarantees on network behavior under prescribed input or parameter perturbations. 
Existing verification methods exploit several forms of model structure: satisfiability and constraint-solving methods give exact certificates for piecewise-linear networks~\cite{ehlers2017formal,katz2017reluplex}, 
mixed-integer formulations encode verification as discrete optimization~\cite{tjeng2019mip}, 
abstract interpretation propagates sound bounds through the network~\cite{gehr2018ai2,singh2019deeppoly}, 
and convex relaxations together with branch-and-bound provide scalable certification algorithms~\cite{bunel2020bab,dvijotham2018dual,wang2021beta,wong2018convex}. %
Recent work~\cite{alexandr2026verification} formulates robustness verification as an algebraic distance minimization problem and studies architecture-dependent complexity measures such as the Euclidean distance degree. 
This points to a broader role for algebraic invariants in exact certification and in quantifying the intrinsic complexity of verification. From this perspective, identifying explicit invariants of attention mechanisms provides a step toward algebraic approaches to verification and robustness in transformer-based~models.

\subsection*{Organization of the paper}

\Cref{sec:setup} introduces lightning self-attention, the coefficient parametrization, and the attention variety. 
\Cref{sec:linear} studies the linear relations among the coefficients. \Cref{sec:single-column} turns to the single-column coefficients, characterizes the corresponding coefficient variety, and derives Chow-type relations together with additional determinantal constraints in the low-rank regime. 
\Cref{sec:cross-column} studies the cross-column variety and derives pencil, flattening, Veronese-type, and Sylvester resultant invariants from this structure. 
Throughout, we include examples illustrating the main constructions and results. 
Finally, \Cref{sec:summary} summarizes the structural results and families of invariants and discusses open problems.

\section{The attention variety of lightning self-attention}\label{sec:setup}

We begin with the standard self-attention construction and then specialize to
\emph{lightning self-attention}, obtained by taking the score transformation to be the identity map. 
This specialization is the one relevant for our
purposes, since it makes the attention map polynomial in the input coordinates and the parameters. 

\vspace{-0.5em}
\begin{definition}[Self-attention module]
Consider matrices $Q,K \in \mathbb{R}^{a \times d}$, $V \in \mathbb{R}^{d' \times d}$, and input tokens $x_i \in \mathbb{R}^d$, $i \in [t]$. 
%
%
The self-attention module associated with weights $W=(Q,K,V)$ is the map
\begin{align}
\varphi_W:\mathbb{R}^{d\times t} &\to \mathbb{R}^{d'\times t},\nonumber\\
(x_i)_{i\in[t]} \quad &\mapsto \quad 
\left(
\sum_{j\in[t]} [\sigma(s_i)]_j\, Vx_j
\right)_{i\in[t]},
\end{align}
where $s_i = ({x_j^\top K^\top Qx_i}/{\sqrt{a}} )_{j\in[t]}\in\mathbb{R}^t$ is the vector of \emph{attention scores} for input token $i\in[t]$, 
and $\sigma\colon \mathbb{R}^t\to\mathbb{R}^t$ is a score transformation.

\end{definition} 

Thus, the $i$th output token is a weighted sum of the \emph{value vectors} $Vx_j$, $j \in [t]$, with attention weights 
$[\sigma(s_i)]_j$, $j \in [t]$. 
In standard self-attention, $\sigma$ is the softmax map.

\subsection{Lightning self-attention}

Lightning self-attention is obtained by replacing the softmax map with the
identity. 
The factor $1/\sqrt{a}$ in the attention scores can be absorbed into $Q$ or $K$, and we omit it throughout the remainder of the paper. 
%
%
Let $X = [x_1 \ \cdots \ x_t] \in \mathbb{R}^{d \times t}$
be the matrix whose columns are the input tokens, 
and define the \emph{attention matrix} 
$A = K^\top Q \in \mathbb{R}^{d \times d}$, $\rank(A)\leq a$. 
Then the lightning self-attention module takes the form 
\[
\varphi_W(X) = VX(X^\top A X) \in \mathbb{R}^{d' \times t}. 
\]
Thus each output coordinate of $\varphi_W(X)$ is a homogeneous cubic polynomial in the entries of $X$, with coefficients depending polynomially on the parameters $W=(Q,K,V)$. Our goal is to characterize the algebraic relations satisfied by these coefficients. 
The following proposition makes this polynomial structure explicit. 

\begin{proposition}
\label{prop:param-ij}
For any $i\in[d']$ and $j\in[t]$, the $(i,j)$th output coordinate is given by
\[
\varphi_W(X)[i,j]
=
\sum_{k,m,l\in[d],\,n\in[t]}
(a_{ml}v_{ik})\,x_{kn}x_{mn}x_{lj}.
\]
\end{proposition}

\begin{proof}
Since $\varphi_W(X)=VX(X^\top AX)$, we have
\[
\varphi_W(X)[i,j]
=
\sum_{k\in[d]} v_{ik}\bigl(X(X^\top AX)\bigr)_{k,j}.
\]
Now
\[
(X^\top AX)_{n,j}
=
\sum_{m,l\in[d]} x_{mn}a_{ml}x_{lj},
\]
so
\[
\bigl(X(X^\top AX)\bigr)_{k,j}
=
\sum_{n\in[t]} x_{kn}(X^\top AX)_{n,j}
=
\sum_{m,l\in[d],\,n\in[t]} a_{ml}x_{kn}x_{mn}x_{lj}.
\]
Multiplying by $v_{ik}$ and summing over $k$ gives the result.
\end{proof}

In the notation above, the index $j$ labels the output token being computed, so we refer to it as the \emph{target position}. 
In a mixed monomial of the form $x_{kn}x_{mn}x_{lj}$ with $n \ne j$, the index $n$ labels another input token contributing to that output, and we call it a \emph{context position}. 
We use this terminology because our arguments track the roles of sequence positions, represented by the columns of the input matrix $X$, 
in contrast to the \emph{query}, \emph{key}, and \emph{value} terminology associated with the parameter matrices $Q$, $K$, and $V$.

For fixed $(i,j)$, \Cref{prop:param-ij} shows that only monomials of the form
$x_{kn}x_{mn}x_{lj}$ can appear. In the next subsection, we organize these
monomials and introduce coefficient coordinates for the resulting cubic
expansions.

\subsection{Monomials and coefficient coordinates}
Our main objective is to determine the algebraic relations satisfied by the coefficients of the lightning self-attention map. By \Cref{prop:param-ij}, each output coordinate $\varphi_W(X)[i,j]$ is a homogeneous cubic polynomial in the entries of $X$, with coefficients depending polynomially on the weight matrices $W=(Q,K,V)$.

To study these relations, we first describe the monomial support of the output polynomials. For each output coordinate $(i,j)\in[d']\times[t]$, let $\mathcal M_{i,j}$ denote the set of cubic monomials that appear in $\varphi_W(X)[i,j]$,~namely
\begin{equation}\label{eq:monomial-supports}
    \mathcal M_{i,j}=\{x_{kn}x_{mn}x_{lj}\mid k,m,l\in[d],\ n\in[t]\}.
\end{equation}

\begin{proposition}
\label{prop:column-support}
The monomial supports $\mathcal M_{i,j}$ satisfy the following properties:
\begin{itemize}
    \item $\mathcal M_{i,j}=\mathcal M_{i',j}$ for all $i,i'\in[d']$ and $j\in[t]$.
    \item $\mathcal M_{i,j}\cap \mathcal M_{i',j'}=\varnothing$ for all $j\neq j'$ and all $i,i'\in[d']$.
\end{itemize}
\end{proposition}

\begin{proof}
The first statement is immediate from the definition of $\mathcal M_{i,j}$ in (\ref{eq:monomial-supports}) which does not depend on $i$.

For the second statement, every monomial in $\mathcal M_{i,j}$ has the form
$x_{kn}x_{mn}x_{lj}$, so its column-index multiset is $\{n,n,j\}$. Thus $j$
is either the unique column index appearing once, or else all three column
indices are equal to $j$. In either case, the column $j$ is determined uniquely
by the monomial. Therefore a monomial cannot lie in both $\mathcal M_{i,j}$ and
$\mathcal M_{i',j'}$ when $j\neq j'$.
\end{proof}

\begin{remark}\label{rmk:reduction-d'=1}
By \Cref{prop:param-ij}, for each fixed output row $i\in[d']$, the
coefficients in that row depend on the attention matrix $A$ and on the $i$th
row of $V$ only. In particular, the coefficient parametrization has the same
form for every output row, up to relabeling the 
rows of $V$. 

Motivated by this observation, we focus primarily on \emph{rowwise
invariants}, i.e., polynomial relations among coefficient variables within a single fixed output row. 
Since every row is governed by the same parametrization, it is enough to derive these relations for $d'=1$ and then restore the row index~$i$. 
In \Cref{sec:linear}, we further reduce to a single fixed output coordinate. \Cref{app:cross-row} briefly discusses cross-row~invariants. 
\end{remark}

The monomial support of lightning self-attention 
occupies only a small subset of the cubic monomials in the variables $x_{ij}$. 
The next proposition makes this sparsity explicit by counting the monomials that can occur within a single output coordinate, and across the entire output matrix.

\begin{proposition}\label{prop:monomial-count}
Let $N_\varphi(d,t)$ denote the total number of distinct cubic monomials
appearing across all output coordinates of $\varphi_W(X)$. Then
\[
N_\varphi(d,t)=\frac{1}{6}dt(d+1)(3dt-2d+2).
\]
The monomial supports associated with distinct target positions $j$ are disjoint and have the same cardinality, while for each fixed target position $j$ the support is identical across output rows $i$. Hence,
each fixed output coordinate $\varphi_W(X)[i,j]$ contains
\[
\frac{N_\varphi(d,t)}{t}
=
\frac{1}{6}d(d+1)(3dt-2d+2)
\]
distinct monomials. 
Since the total number of cubic monomials in the polynomial ring
$\mathbb{R}[x_{ij}]_{i\in[d],\,j\in[t]}$ is
\[
N(d,t)=\binom{dt+2}{3}=\frac{dt(dt+1)(dt+2)}{6},
\]
we have that for fixed $d$,
\[
\frac{N_\varphi(d,t)}{N(d,t)}
=
\frac{(d+1)(3dt-2d+2)}{(dt+1)(dt+2)}
=
\frac{3(d+1)}{d}\frac{1}{t}
+
O\left(\frac{1}{t^2}\right)
\qquad\text{as } t\to\infty.
\]
\end{proposition}

\begin{proof}
By \Cref{prop:column-support}, the supports for different output
columns are disjoint, and the supports within a fixed output column are the
same across output rows. It is therefore enough to count the monomials in one
fixed output coordinate $\varphi_W(X)[i,j]$, and then multiply by $t$.

Fix $j\in[t]$. Every monomial in $\mathcal M_{i,j}$ has the form
$x_{kn}x_{mn}x_{lj}$ with $k,m,l\in[d]$ and $n\in[t]$. There are two cases. If $n=j$, all three variables lie in the $j$th column, so the number of
distinct monomials is the number of multisets of size $3$ from $[d]$, namely
$\binom{d+2}{3}$. If $n\neq j$, then two variables lie in column $n$ and one lies in column $j$.
There are $t-1$ choices for $n$, $\binom{d+1}{2}$ choices for the unordered
pair $\{k,m\}$, and $d$ choices for~$l$. Thus this contributes
$(t-1)\binom{d+1}{2}d$ monomials. Therefore
\[
|\mathcal M_{i,j}|
=
\binom{d+2}{3}+(t-1)\binom{d+1}{2}d
=
\frac{1}{6}d(d+1)(3dt-2d+2).
\]
Multiplying by $t$ gives
\[
N_\varphi(d,t)=t\,|\mathcal M_{i,j}|=\frac{1}{6}dt(d+1)(3dt-2d+2).
\]

Finally, the total number of cubic monomials in $dt$ variables is
$\binom{dt+2}{3}$, so the stated ratio follows. Thus, for fixed $d$, only a proportion on the order of $1/t$ of all cubic
monomials can appear in the output of lightning self-attention.
\end{proof}

We now introduce coefficient coordinates for the monomials in
$\mathcal M_{i,j}$. 
Let $\operatorname{Mult}_r([d])$ denote the multisets of
cardinality $r$ in $[d]$, and, for a multiset $\mathcal A$, let $\operatorname{Perm}(\mathcal{A})$ denote the set of distinct ordered tuples obtained by permuting its elements. 
For each $\mathcal K\in \operatorname{Mult}_3([d])$,
we write $\prod_{u\in \mathcal K} x_{uj}$ for the corresponding single-column
monomial in the $j$th column. For each $\mathcal A\in
\operatorname{Mult}_2([d])$, $b\in[d]$, and $n\neq j$, we write
$\left(\prod_{u\in \mathcal A} x_{un}\right)x_{bj}$ for the corresponding
cross-column monomial. Thus, for each output coordinate $(i,j)\in[d']\times[t]$, we may~write
\[
\varphi_W(X)[i,j]
=
\sum_{\mathcal K\in \operatorname{Mult}_3([d])}
c^{[i]}_{j}(\mathcal K)\prod_{u\in \mathcal K} x_{uj}
+
\sum_{n\neq j}
\sum_{\mathcal A\in \operatorname{Mult}_2([d])}
\sum_{b\in[d]}
c^{[i]}_{n,j}(\mathcal A,b)
\left(\prod_{u\in \mathcal A} x_{un}\right)x_{bj},
\]
where $c^{[i]}_{j}(\mathcal K)$ and $c^{[i]}_{n,j}(\mathcal A,b)$ denote the coefficients in the standard monomial basis.

\begin{proposition}
\label{prop:sums-monomials}
The standard polynomial coefficients are given as follows:
\begin{enumerate}
    \item For $\mathcal K\in \operatorname{Mult}_3([d])$,
    \[
    c^{[i]}_{j}(\mathcal K)
    =
    \sum_{(p_1,p_2,p_3)\in \operatorname{Perm}(\mathcal K)}
    a_{p_2p_3}v_{ip_1}.
    \]

    \item For $\mathcal A\in \operatorname{Mult}_2([d])$, $b\in[d]$, and $n\neq j$,
    \[
    c^{[i]}_{n,j}(\mathcal A,b)
    =
    \sum_{(p_1,p_2)\in \operatorname{Perm}(\mathcal A)}
    a_{p_2b}v_{ip_1}.
    \]
\end{enumerate}
\end{proposition}

\begin{proof}
By \Cref{prop:param-ij}, the ordered expansion of
$\varphi_W(X)[i,j]$ is
\[
\varphi_W(X)[i,j]
=
\sum_{k,m,l\in[d],\,n\in[t]}
(a_{ml}v_{ik})\,x_{kn}x_{mn}x_{lj}.
\]
To obtain the standard polynomial coefficients, we group together all ordered
terms that yield the same monomial. If $n=j$, then all three variables lie in column $j$, so the monomial is of the
form $\prod_{u\in\mathcal K} x_{uj}$ for some
$\mathcal K\in\operatorname{Mult}_3([d])$. The ordered triples
$(k,m,l)$ that produce this monomial are precisely the distinct permutations of
$\mathcal K$. Summing the corresponding coefficients gives
\[
c^{[i]}_{j}(\mathcal K)
=
\sum_{(p_1,p_2,p_3)\in\operatorname{Perm}(\mathcal K)}
a_{p_2p_3}v_{ip_1}.
\]
If $n\neq j$, then $x_{lj}$ is the unique factor in column $j$, so its row index
$b=l$ is fixed, while the two factors in column $n$ determine a multiset
$\mathcal A=\{k,m\}\in\operatorname{Mult}_2([d])$. Thus the ordered pairs
$(k,m)$ that produce the monomial
$\left(\prod_{u\in\mathcal A}x_{un}\right)x_{bj}$
are precisely the distinct permutations of $\mathcal A$. Summing the
corresponding coefficients gives
\[
c^{[i]}_{n,j}(\mathcal A,b)
=
\sum_{(p_1,p_2)\in\operatorname{Perm}(\mathcal A)}
a_{p_2b}v_{ip_1}.
\]
\end{proof}

\subsection{Scaled coefficients and the attention variety}

While $c_j^{[i]}(\mathcal K)$ and $c_{n,j}^{[i]}(\mathcal A,b)$ are the standard polynomial coefficients, 
it is useful for our analysis to arrange
them into symmetric tensors. 
To ensure 
that contraction of these tensors recovers the corresponding polynomials, 
we 
normalize each coefficient by the size of its permutation orbit. Accordingly, we introduce 
scaled coefficient coordinates  
%
\begin{equation}
\label{eq:scaled-coefficients}
y^{[i]}_{j}(\mathcal K)
=
\frac{1}{|\operatorname{Perm}(\mathcal K)|}\,c^{[i]}_{j}(\mathcal K),
\qquad
y^{[i]}_{n,j}(\mathcal A,b)
=
\frac{1}{|\operatorname{Perm}(\mathcal A)|}\,c^{[i]}_{n,j}(\mathcal A,b).
\end{equation}
As before, these are indexed by $\mathcal K\in \operatorname{Mult}_3([d])$, $\mathcal A\in \operatorname{Mult}_2([d])$, and $b\in[d]$. 
We call 
$y^{[i]}_{j}(\mathcal K)$ the \textit{single-column coefficients}, and 
$y^{[i]}_{n,j}(\mathcal A,b)$ the \textit{cross-column coefficients}. 
%
%
We now package these scaled coefficients into a polynomial parametrization, and denote their joint value at $W$ by $y(W)$. 

\begin{definition}[Parametrization map]
\label{def:parametrization}
Since each output row has $N_\varphi(d,t)$ coefficient coordinates, the
scaled coefficient array $y(W)$ determines a point of
$\mathbb{R}^{d'N_\varphi(d,t)}$. We define the parametrization map
$$
\mu:\mathbb{R}^{a\times d}\times \mathbb{R}^{a\times d}\times
\mathbb{R}^{d'\times d}
\to \mathbb{R}^{d'N_\varphi(d,t)},\qquad
W=(Q,K,V)\mapsto y(W).
$$
\end{definition} 

\noindent Since $\mu$ has real coefficients, we use the same notation for its complexification.

\begin{definition}[Attention variety]
\label{def:attention-variety}
The \emph{attention variety} is the Zariski closure
$$
\mathcal{V}:=\overline{\operatorname{im}(\mu)}
\subseteq \mathbb{C}^{d'N_\varphi(d,t)}.
$$
We write $J\subseteq \mathbb{C}[y]$ for its vanishing ideal.
\end{definition}

\section{Linear relations}
\label{sec:linear} 

The coefficients of the lightning self-attention map satisfy two types of linear relations. 
The first are \emph{sequence-copy relations}, which identify groups of coefficient variables in the output coefficient array 
that coincide when the column indices of the input matrix $X$ are varied. 
The second are \emph{internal symmetrization relations}, which occur within a fixed output coordinate when $t\geq2$ 
and express each single-column coefficient as a linear combination of cross-column coefficients. 

\paragraph{Sequence-copy relations.}
Fix $i\in[d']$. By \Cref{prop:sums-monomials}, under the
parametrization map $\mu$ the coordinate $y^{[i]}_j(\mathcal K)$ takes a value
that depends only on the multiset $\mathcal K\in \operatorname{Mult}_3([d])$
and not on the output column $j$. Likewise, the coordinate
$y^{[i]}_{n,j}(\mathcal A,b)$ depends only on
$\mathcal A\in \operatorname{Mult}_2([d])$ and $b\in[d]$, and is independent of
the choice of $(n,j)$ subject to $n\neq j$. Hence the defining ideal $J$
contains the linear relations
\begin{equation*}
y^{[i]}_{j}(\mathcal K)-y^{[i]}_{j'}(\mathcal K)=0
\qquad
\text{for all } \mathcal K \in \operatorname{Mult}_3([d]),\ j,j' \in [t],
\end{equation*}
and
\begin{equation*}
y^{[i]}_{n,j}(\mathcal A,b)-y^{[i]}_{n',j'}(\mathcal A,b)=0
\qquad
\text{for all } \mathcal A \in \operatorname{Mult}_2([d]),\ b \in [d],\
n\neq j,\ n'\neq j'.
\end{equation*}
These relations identify repeated coefficient blocks obtained by varying the
column indices.

\paragraph{Internal symmetrization relations.}
When $t\geq2$, after quotienting out the sequence-copy relations, additional linear invariants already appear inside a single fixed output coordinate $(i,j)$. They arise because the coefficient of a single-column monomial is a sum over all permutations of the three row indices in $\mathcal K$, whereas the coefficient of a cross-column monomial is a sum only over permutations of the two row indices attached to the context column $n$ (that is, the repeated input column), with the target column $j$ fixed. 
As a result, every single-column coefficient can be written as an exact linear combination of cross-column coefficients.

\begin{proposition}[Linear invariants]
\label{prop:symmetrization}
Assume $t\geq2$. Fix an output coordinate $(i,j)\in[d']\times[t]$ and a context column
$n\in[t]$ with $n\neq j$. Let $\mathcal K\in\operatorname{Mult}_3([d])$.
Then the defining ideal $J$ contains the linear polynomial
\[
|\operatorname{Perm}(\mathcal K)|\,y^{[i]}_j(\mathcal K)
-
\sum_{\substack{\mathcal A\cup\{b\}=\mathcal K\\|\mathcal A|=2}}
|\operatorname{Perm}(\mathcal A)|\,y^{[i]}_{n,j}(\mathcal A,b),
\]
where the sum runs over all distinct decompositions of the multiset
$\mathcal K$ into a multiset $\mathcal A$ of size $2$ and a singleton $\{b\}$.
\end{proposition}

\begin{proof}
By \Cref{prop:sums-monomials}, the unscaled coefficient of the
single-column monomial indexed by $\mathcal K$ is
\[
c_j^{[i]}(\mathcal K)
=
\sum_{(p_1,p_2,p_3)\in\operatorname{Perm}(\mathcal K)}
a_{p_2p_3}v_{ip_1}.
\]
Partition this sum according to the value of the third index $p_3=b$. For each
such choice, the remaining two indices form a multiset
$\mathcal A\in\operatorname{Mult}_2([d])$ with
$\mathcal A\cup\{b\}=\mathcal K$. Hence
\[
c_j^{[i]}(\mathcal K)
=
\sum_{\substack{\mathcal A\cup\{b\}=\mathcal K\\|\mathcal A|=2}}
\left(
\sum_{(p_1,p_2)\in\operatorname{Perm}(\mathcal A)}
a_{p_2b}v_{ip_1}
\right).
\]
By \Cref{prop:sums-monomials}, the inner sum is exactly
$c_{n,j}^{[i]}(\mathcal A,b)$. Therefore
\[
c_j^{[i]}(\mathcal K)
=
\sum_{\substack{\mathcal A\cup\{b\}=\mathcal K\\|\mathcal A|=2}}
c_{n,j}^{[i]}(\mathcal A,b).
\]
Using the definition of the parametrization map, the ambient variables
$y^{[i]}_j(\mathcal K)$ and $y^{[i]}_{n,j}(\mathcal A,b)$ evaluate to the
corresponding scaled coefficients, so
\[
c_j^{[i]}(\mathcal K)=|\operatorname{Perm}(\mathcal K)|\,y^{[i]}_j(\mathcal K),
\qquad
c_{n,j}^{[i]}(\mathcal A,b)=|\operatorname{Perm}(\mathcal A)|\,y^{[i]}_{n,j}(\mathcal A,b).
\]
Substituting shows that the displayed linear polynomial vanishes on
$\operatorname{im}(\mu)$, and therefore belongs to~$J$.
\end{proof}

\begin{corollary}\label{cor:linear-invariants-count}
Assume $t\geq2$. For a single output coordinate $(i,j)$, the defining ideal of the attention variety contains $\binom{d+2}{3}$ linearly independent linear invariants arising from internal symmetrization.
\end{corollary}

\begin{proof}
Fix any context column $n \neq j$; by \Cref{prop:symmetrization}, each $\mathcal K \in \mathrm{Mult}_3([d])$ determines such a relation for any fixed $n$, and different choices of $n$ only produce sequence-copy equivalents. 

There are exactly $\binom{d+2}{3}$ single-column monomials in the coordinate~$(i,j)$. By \Cref{prop:symmetrization}, each such monomial indexed by $\mathcal K\in \operatorname{Mult}_3([d])$ gives a linear relation whose single-column term is
$
|\operatorname{Perm}(\mathcal K)|\, y^{[i]}_j(\mathcal K).
$
As $\mathcal K$ varies, these single-column variables are distinct, and each appears only in the relation indexed by that same $\mathcal K$. Therefore the resulting $\binom{d+2}{3}$ relations are linearly independent.
\end{proof}

\begin{remark}[Reduction to one representative family]\label{rmk:reduction-to-one-coordinate}
Two independent reductions allow us to restrict to one representative family of coefficient variables for the rowwise linear relations studied in this section. First, by Remark~\ref{rmk:reduction-d'=1}, it is enough to take $d'=1$. Indeed, every rowwise relation involving the variables $y^{[1]}_\bullet$ has an identical copy involving the variables $y^{[i]}_\bullet$ for any other row $i$. Second, for fixed $i$, the sequence-copy relations identify coefficient families obtained by varying the column indices. Thus all single-column families are copies of one another, and, when $t\geq2$, all cross-column families are likewise copies. Therefore, after applying the reduction $d'=1$, it suffices to work with one fixed output coordinate $(1,j)$. When $t\geq2$, the choice of context column $n\neq j$ is irrelevant for the linear relations, because different choices of $n$ give the same relation after replacing the corresponding variables using the sequence-copy relations.
\end{remark}

\paragraph{Convention.}
In view of \Cref{rmk:reduction-d'=1} and
\Cref{rmk:reduction-to-one-coordinate}, we henceforth set $d'=1$ and work
with a fixed output coordinate $(1,j)$.
Accordingly, we write $V=v^\top$, where
$v=(v_1,\ldots,v_d)^\top\in\mathbb R^d$. For conceptual
statements, we write
\[
y_j(\mathcal K), \qquad y_{n,j}(\mathcal A,b)\ \ (n\neq j),
\]
where $\mathcal K \in \operatorname{Mult}_3([d])$, $\mathcal A \in
\operatorname{Mult}_2([d])$, and $b \in [d]$. These correspond to the ambient
coordinates
\[
y_{(k_1,j),(k_2,j),(k_3,j)}
\qquad\text{and}\qquad
y_{(a_1,n),(a_2,n),(b,j)},
\]
where in the cross-column case the two context-column indices are written first and the target-column index is written last. 
Thus, for example, 
\[
y_1(\{1,2,2\})=y_{(1,1),(2,1),(2,1)},
\qquad
y_{2,1}(\{1,2\},2)=y_{(1,2),(2,2),(2,1)}.
\]
By abuse of notation, we continue to write $J$ for the corresponding ideal in
the reduced ambient variables.

\begin{example}
Let $d=t=a=2$ and $d'=1$. For this example, fix the target column $j=1$ and
choose the context column $n=2$. Then the defining ideal $J$ contains four
linear polynomials and one quartic polynomial. The quartic invariant is derived
in \Cref{ex:quartic}. By \Cref{prop:symmetrization}, the four
linear generators are
\begin{align*}
    & y_{(2,1),(2,1),(2,1)} - y_{(2,2),(2,2),(2,1)},\\
    & 3\,y_{(1,1),(2,1),(2,1)} - 2\,y_{(1,2),(2,2),(2,1)} - y_{(2,2),(2,2),(1,1)},\\
    & 3\,y_{(1,1),(1,1),(2,1)} - y_{(1,2),(1,2),(2,1)} - 2\,y_{(1,2),(2,2),(1,1)},\\
    & y_{(1,1),(1,1),(1,1)} - y_{(1,2),(1,2),(1,1)}. 
\end{align*}
\end{example}

\begin{example}
Let $d=3$, $t=2$, $d'=1$, and $a=1$. By \Cref{prop:symmetrization}
and \Cref{cor:linear-invariants-count}, there are
$\binom{3+2}{3}=10$ linearly independent linear relations in a single fixed
output coordinate. However, these linear relations do not
generate the full defining ideal $J$; it is minimally
generated by these $10$ linear relations together with $45$ quadratic and $10$
cubic polynomials. We construct these nonlinear generators explicitly in \Cref{ex:3121}.
\end{example}

\section{Single-column monomials}\label{sec:single-column} 

We now turn to the single-column coefficients, that is, the cubic forms supported in a single input column. By restricting the attention map to terms supported entirely in the target column $j$, the associated cubic has the form
$
f(x)=(v^\top x)(x^\top A x),
$
where $\rank(A)\leq a$. Thus every single-column cubic factors as a product of a linear form and a quadratic form, with an additional rank constraint inherited from the attention matrix $A$. 
We first describe this geometric structure and then derive algebraic equations from~it. 

The results of this section describe the single-column slice for any
sequence length~$t$, but they are primarily needed in the single-token
case $t=1$. When $t\geq2$, the linear relations of \Cref{sec:linear}
express the single-column coefficients in terms of the cross-column
coefficients, so the nonlinear single-column relations are redundant.

For the structural description of the single-column cubics, we work over
$\CC$. For a quadratic form $q$, we write $\rank(q)$ for the rank of its
symmetric coefficient matrix.

\begin{proposition}[Single-column structure]
\label{prop:single-column-structure}
The single-column cubics arising from the attention parametrization are
exactly the cubics of the form
$
f=\ell q,
$
where $\ell\in\CC[x]_1$, $q\in\CC[x]_2$, and
$
\rank(q)\leq 2a.
$
\end{proposition}

\begin{proof}
By restricting \Cref{prop:param-ij} to the terms supported entirely in
the target column $j$, the associated cubic is
$$
f(x)=(v^\top x)(x^\top A x),
$$
where $\rank(A)\leq a$. Thus $f=\ell q$ with
$\ell=v^\top x$ and $q=x^\top A x$. The quadratic form $q$ is
represented by the symmetric matrix
$
\frac{1}{2}(A+A^\top),
$
so
$$
\rank(q)
\leq
\rank(A)+\rank(A^\top)
\leq 2a.
$$

Conversely, suppose that $f=\ell q$ with $\rank(q)=r\leq 2a$.
Over $\CC$, after a linear change of coordinates we may write
$$
q=z_1^2+\cdots+z_r^2.
$$
Pairing the terms gives
$$
z_{2k-1}^2+z_{2k}^2
=
(z_{2k-1}+iz_{2k})(z_{2k-1}-iz_{2k}),
$$
with a final square $z_r^2$ if $r$ is odd. Hence $q$ is a sum of at
most $\lceil r/2\rceil\leq a$ products of linear forms:
$$
q=\sum_{k=1}^a \ell_km_k,
$$
where some terms may be zero. Writing
$\ell_k=u_k^\top x$ and $m_k=w_k^\top x$ and setting
$$
A=\sum_{k=1}^a u_kw_k^\top,
$$
we have $\rank(A)\leq a$ and $q=x^\top A x$. Choosing $v$ such that
$\ell=v^\top x$ realizes $f$ as a single-column cubic of the attention
parametrization.
\end{proof}

We next derive algebraic equations from this factorization. One such
construction uses linear syzygies among the first partial derivatives of
the cubic.

\begin{definition}[Syzygy]
Let $R = \mathbb{R}[x_1, \dots, x_d]$ be a polynomial ring, and let $g_1, \dots, g_k \in R$ be a collection of polynomials. A \emph{syzygy} among the $g_i$ is a tuple of polynomials $(p_1, \dots, p_k) \in R^k$ such that
\begin{equation*}
    p_1 g_1 + p_2 g_2 + \dots + p_k g_k = 0.
\end{equation*}
A syzygy is called \emph{trivial} if all $p_i = 0$. If the 
$p_i$ are linear forms, we call it a \emph{linear syzygy}. 
\end{definition}

In our setting, the relevant collection of polynomials is the set of first partial derivatives of a cubic form~$f$, namely $f_{x_1}, \dots, f_{x_d}$. We encode linear syzygies among these derivatives using the natural action of the Lie algebra $\mathfrak{sl}_d$ on $f$. More concretely, an element $M=(m_{uv}) \in \mathfrak{sl}_d$ is a trace-zero $d \times d$ matrix, and it acts on $f$ by sending it to the polynomial $M \cdot f = \sum_{u,v} m_{uv} x_u f_{x_v}$. Thus $M \cdot f$ is a linear combination of the first partial derivatives of $f$, with coefficients that are linear forms in the variables $x_1,\dots,x_d$. Therefore, the condition $M \cdot f = 0$ is exactly a linear syzygy among $f_{x_1},\dots,f_{x_d}$. 

Let $U = \mathbb{R}^d$, let $U^*$ denote its dual vector space, and let $S^3U^*$ denote the third symmetric power of $U^*$, naturally identified with the space of homogeneous cubic polynomials on $U$.
Let $f \in S^3 U^*$ be a homogeneous cubic polynomial in variables $x_1, \dots, x_d$. 
The Lie algebra $\mathfrak{gl}_d$ of all $d\times d$ real matrices
acts on $f$ via first-order differential operators
\[
E_{uv}=x_u\frac{\partial}{\partial x_v}.
\]
Since $\mathfrak{gl}_d = \mathfrak{sl}_d \oplus \mathbb{R}I$,
every element of $\mathfrak{gl}_d$ decomposes uniquely into a trace-zero part and a scalar multiple of the identity. Indeed, any $A\in\mathfrak{gl}_d$ can be written uniquely as
\[
A = \left(A - \frac{\operatorname{tr}(A)}{d}I\right)
+ \frac{\operatorname{tr}(A)}{d}I,
\]
with the first term in $\mathfrak{sl}_d$. For homogeneous cubics, the identity matrix acts by the Euler operator
\[
E=\sum_{u=1}^d x_u \frac{\partial}{\partial x_u}.
\]
Euler's identity gives $E(f)=3f$. So the $\mathbb{R}I$-summand contributes only scalar multiples of $f$, so we restrict to the trace-zero subalgebra $\mathfrak{sl}_d$, which has dimension~$d^2-1$. 

\begin{definition}[Lie algebra flattening matrix]
We define the linear map
\[
\Psi_f:\mathfrak{sl}_d \to S^3U^*, \qquad D \mapsto D(f).
\]
Fix a standard basis of $\mathfrak{sl}_d$ given by the off-diagonal operators $E_{uv}$ for $u\neq v$ and the trace-zero diagonal operators $H_u=E_{uu}-E_{dd}$ for $u<d$, together with the standard monomial basis of $S^3U^*$. The \emph{Lie algebra flattening matrix} $M_{\mathrm{Lie}}(f)$ is the $\binom{d+2}{3}\times(d^2-1)$ matrix representing $\Psi_f$ in these bases.
\end{definition}

The kernel of $M_{\mathrm{Lie}}(f)$ consists exactly of the trace-zero differential operators $D \in \mathfrak{sl}_d$ such that $D(f)=0$. If we write
\[
D=\sum_{u,v} m_{uv}\,x_u\frac{\partial}{\partial x_v},
\]
then
\[
D(f)=\sum_{u,v} m_{uv}\,x_u\frac{\partial f}{\partial x_v}
=\sum_{u,v} m_{uv}\,x_u f_{x_v}
=\sum_{v=1}^d \left(\sum_{u=1}^d m_{uv}x_u\right) f_{x_v}.
\]
Thus $D(f)=0$ takes the form
\[
\sum_{v=1}^d p_v\,f_{x_v}=0,
\qquad\text{where}\qquad
p_v=\sum_{u=1}^d m_{uv}x_u
\]
is a linear form for each $v$. Therefore, a nonzero operator $D \in \ker M_{\mathrm{Lie}}(f)$ gives a non-trivial linear syzygy among the first partial derivatives of $f$.

The next lemma shows that cubic forms in $d\geq 3$ variables that are reducible of type $(2,1)$, that is, cubics that factor as the product of a linear and a quadratic form, have rank-deficient Lie algebra flattening matrices. 
For $d\leq 2$, this factorization condition imposes no constraint, since every cubic form admits such a factorization over $\mathbb{R}$.

\begin{lemma} \label{lem:syzygy-factorization}
Let $d \ge 3$. If a cubic polynomial $f \in \mathbb{R}[x_1, \dots, x_d]$ factors as $f = \ell \cdot q$ for a linear form $\ell$ and a quadratic form $q$, then the map $\Psi_f$ has a non-trivial kernel. Consequently, the matrix $M_{\mathrm{Lie}}(f)$ has rank strictly less than $d^2-1$, and all of its maximal $(d^2-1) \times (d^2-1)$ minors vanish. 
\end{lemma}

\begin{proof}
The rank of the Lie algebra flattening matrix $M_{\mathrm{Lie}}(f)$ and the dimension of its kernel are invariant under invertible linear changes of coordinates, since such a change induces invertible row and column operations on the matrix representing $\Psi_f$. Therefore, we may choose coordinates $(x_1, \dots, x_d)$ such that the linear form is simply $\ell = x_1$. In these adapted coordinates, our polynomial is $f = x_1 q(x_1, \dots, x_d)$. Taking partial derivatives with respect to the variables $x_k$ for $k \ge 2$, the product rule gives $\partial_k f = x_1 \partial_k q$ for all $k \ge 2$. Since $d\ge 3$, there are at least two indices $k\ge 2$. We therefore consider the family of linear forms $\{\partial_k q\}_{k\ge 2}$ and distinguish two cases according to the dimension of its span.

\textbf{Case 1:} Suppose the span of $\{\partial_k q\}_{k\ge 2}$ has dimension at least $2$. Relabeling indices among $k\ge2$ if necessary, the linear forms $\partial_2 q$ and $\partial_3 q$ are linearly independent over $\mathbb{R}$. We construct the differential~operator
\[
D = (\partial_3 q)\partial_2 - (\partial_2 q)\partial_3.
\]
Applying $D$ to $f$ gives
\[
D(f) = (\partial_3 q)(\partial_2 f) - (\partial_2 q)(\partial_3 f)
= (\partial_3 q)(x_1 \partial_2 q) - (\partial_2 q)(x_1 \partial_3 q) = 0.
\]
Because the coefficients $\partial_3 q$ and $\partial_2 q$ are linear forms, $D$ is a first-order differential operator with linear coefficients, hence $D \in \mathfrak{gl}_d$. Moreover, $D \neq 0$ since $\partial_2 q$ and $\partial_3 q$ are linearly independent. To show that $D \in \mathfrak{sl}_d$, write $D=\sum_{v=1}^d p_v\,\partial_v$ with linear coefficients $p_v$. Such an operator corresponds to a matrix in $\mathfrak{gl}_d$, and its trace is given by
\[
\operatorname{tr}(D)=\sum_{v=1}^d \partial_v p_v,
\]
since differentiating each linear coefficient picks out the corresponding diagonal entry. For the operator above, the only nonzero coefficients are $p_2=\partial_3 q$ and $p_3=-\partial_2 q$. Thus, by symmetry of mixed partial~derivatives,
\[
\operatorname{tr}(D) = \frac{\partial}{\partial x_2}(\partial_3 q) + \frac{\partial}{\partial x_3}(-\partial_2 q)
= \partial_2\partial_3 q - \partial_3\partial_2 q = 0.
\]
Thus, $D$ is a nonzero trace-zero annihilator of $f$.

\textbf{Case 2:} Suppose the span of $\{\partial_k q\}_{k\ge 2}$ has dimension at most $1$. Since there are at least two indices $k\ge2$, it follows that the family $\{\partial_k q\}_{k\ge2}$ is linearly dependent over $\mathbb{R}$. Thus there exists a non-trivial linear combination of scalars $c_k$ such that $\sum_{k=2}^d c_k \partial_k q = 0$. Using $\partial_k f=x_1\partial_k q$, this gives
\begin{equation*}
    \sum_{k=2}^d c_k \partial_k f = x_1 \sum_{k=2}^d c_k \partial_k q = 0.
\end{equation*}
Define $D = x_1 \sum_{k=2}^d c_k \partial_k$. Then $D(f)=0$, and $D \neq 0$ since not all $c_k$ vanish. Furthermore, $D$ has trace~zero: 
\begin{equation*}
    \operatorname{tr}(D) = \sum_{k=2}^d c_k \frac{\partial x_1}{\partial x_k} = 0.
\end{equation*}
Thus, $D$ is again a non-zero, trace-zero annihilator of $f$.

In either case, there exists a non-trivial operator $D \in \mathfrak{sl}_d$ in the kernel of $\Psi_f$. Therefore, $M_{\mathrm{Lie}}(f)$ has non-trivial kernel, so its rank is less than $d^2-1$. Consequently, all of its maximal $(d^2-1) \times (d^2-1)$ minors vanish on this locus.
\end{proof}

By Proposition~\ref{prop:single-column-structure}, every single-column cubic is reducible of type $(2,1)$. 
%
Since this factorization property is invariant under nonzero scaling, it is natural to consider these cubics projectively. 
For a nonzero cubic $f\in S^3U^*$, write $[f]$ for its class in
$\mathbb{P}(S^3U^*)$. 
The 
variety of reducible cubic forms of type $(2,1)$ is the projective image of 
$$
\mathbb{P}(U^*)\times\mathbb{P}(S^2U^*)
\longrightarrow
\mathbb{P}(S^3U^*),
\qquad
([\ell],[q])\longmapsto[\ell q],
$$
see \cite[Section 4.5]{math6120314}. 
We refer to this as the \emph{$(2,1)$ Chow-type variety}, in analogy to the Chow variety of completely decomposable forms \cite[Chapter 8.6]{landsberg2012tensors}. 


\begin{corollary}
\label{cor:single-column-chow}
Assume $d\geq3$. If $f$ is a nonzero single-column cubic, then $[f]$
lies on the $(2,1)$ Chow-type variety. Consequently, the Lie
algebra flattening matrix $M_{\mathrm{Lie}}(f)$ is rank-deficient, and
all of its maximal $(d^2-1)\times(d^2-1)$ minors vanish.
\end{corollary}

\begin{proof}
The factorization follows from \Cref{prop:single-column-structure}, and
the rank deficiency of $M_{\mathrm{Lie}}(f)$ follows from
\Cref{lem:syzygy-factorization}.
\end{proof}

\begin{remark}
In the tensor geometry literature, equations for Chow varieties are
often obtained from Koszul flattenings \cite{landsberg2012tensors}.
Here we use a different construction, namely the Lie algebra flattening
$M_{\mathrm{Lie}}(f)$ arising from the infinitesimal
$\mathfrak{sl}_d$-action on $S^3U^*$. Its kernel records linear
syzygies among the first partial derivatives of $f$, and
\Cref{lem:syzygy-factorization} shows that the split structure of the
single-column cubics forces this kernel to be nontrivial.

\end{remark}

The attention parametrization also imposes a more specific rank
constraint on the single-column coefficient tensor, depending on the rank of $A$. 

\begin{proposition}[Single-column rank constraint]
\label{prop:rank-bounds}
Let $T$ be the symmetric tensor of scaled single-column coefficients,
where $T_{k_1,k_2,k_3} = y_j(\{k_1,k_2,k_3\})$ is the scaled coefficient of the unordered monomial $x_{k_1j}x_{k_2j}x_{k_3j}$. 
Define
$$
N
=
(T_{k,k,k'})_{k,k'\in[d]}
\in\CC^{d\times d}.
$$
Then
$$
\rank N\leq 2\rank(A)+1.
$$
In particular, if $2a+1<d$, then all $(2a+2)\times(2a+2)$ minors of
$N$ vanish.
\end{proposition}

\begin{proof}
Evaluating the coefficients gives
\begin{align*}
N
&=
\frac13(a_{kk'}v_k)_{k,k'\in[d]}
+
\frac13(a_{k'k}v_k)_{k,k'\in[d]}
+
\frac13(a_{kk}v_{k'})_{k,k'\in[d]}\\
&=
\frac13\diag(v)A
+
\frac13\diag(v)A^\top
+
\frac13\diag(A)v^\top.
\end{align*}
Here $\diag(A)$ denotes the column vector of diagonal entries of $A$,
while $\diag(v)$ denotes the diagonal matrix with diagonal $v$.
Therefore, by subadditivity of rank,
$$
\rank N
\leq
\rank(\diag(v)A)
+
\rank(\diag(v)A^\top)
+
\rank(\diag(A)v^\top)
\leq
2\rank(A)+1.
$$
\end{proof}

\begin{example}[$d=a=3$, $d'=t=1$]
\label{ex:type2-relation}
Consider the smallest nontrivial case $d=a=3$ and $d'=t=1$. 
The single-column cubic is of split type $(2,1)$, so
\Cref{cor:single-column-chow} implies that its $10\times8$ Lie algebra
flattening matrix $M_{\mathrm{Lie}}(f)$ is rank-deficient. 
Hence all forty-five $8\times8$ maximal minors vanish. 
A \texttt{Macaulay2} computation shows that the defining ideal has
codimension $2$ and is minimally generated by $35$ degree-$8$ maximal
minors of $M_{\mathrm{Lie}}(f)$. 
The explicit coefficient parametrization, the Lie algebra flattening matrix, and the computation are given in \Cref{app:single-column-example}. Since $2a+1 = 7 > 3 =d$, \Cref{prop:rank-bounds} yields no additional
determinantal constraints.
\end{example}

\begin{remark}
The rank-deficiency condition provided by the Lie algebra flattening is necessary, but not sufficient, for the factorization in \Cref{prop:single-column-structure}.
Indeed, its kernel consists of trace-zero linear differential operators
that annihilate $f$:
$$
\ker(\Psi_f)
=
\mathfrak{stab}(f)\cap\mathfrak{sl}_d,
$$
where
$$
\mathfrak{stab}(f)
=
\{D\in\mathfrak{gl}_d:D(f)=0\}.
$$
Factorization as in \Cref{prop:single-column-structure} guarantees the
existence of such an operator, so 
rank deficiency of
$M_{\mathrm{Lie}}(f)$ is a necessary condition for membership in the
single-column variety. 
However, such operators may also arise for other
reasons.

For example, when $d=4$, the cubic
$
f=x_1^3+x_2^3+x_3^3
$
is independent of $x_4$. Hence the nonzero operator
$
x_1\frac{\partial}{\partial x_4}\in\mathfrak{sl}_4
$
annihilates $f$, so $M_{\mathrm{Lie}}(f)$ drops rank. In fact,
$\rank M_{\mathrm{Lie}}(f)=12$, whereas the generic rank is $15$.
However, $f$ is irreducible, so it has no linear factor and therefore
does not belong to the single-column variety. Thus, the Lie algebra
flattening equations generally cut out a strictly larger locus. 
\end{remark}

\begin{corollary}
\label{cor:rank-one-single-column}
If $a=1$, then the single-column variety is the affine cone over the
Chow variety of completely reducible cubic forms. Equivalently, every
nonzero single-column cubic factors as a product of three linear forms,
and every product of three linear forms arises in this way.
\end{corollary}

\begin{proof}
By \Cref{prop:single-column-structure}, every single-column cubic has
the form $f=\ell q$ with $\rank(q)\leq2$. Over $\CC$, every nonzero
quadratic form of rank at most $2$ factors as a product of two linear
forms. Hence $f$ is a product of three linear forms.

Conversely, given three linear forms $\ell_1,\ell_2,\ell_3$, choose
$v$ so that $v^\top x=\ell_1$ and choose a rank-one matrix $A$ such
that
$
x^\top A x=\ell_2\ell_3.
$
Then the corresponding single-column cubic is
$f=\ell_1\ell_2\ell_3$.
\end{proof}

\begin{example}\label{ex:type1-relation}
Let $d=4$, $t=1$, and $a=1$, and let
$\mathcal V:=\overline{\operatorname{im}(\mu)}$ denote the
corresponding attention variety. By
\Cref{cor:rank-one-single-column}, $\mathcal V$ is the affine cone over
the Chow variety of completely reducible cubic forms in four
variables. Since
$\dim\mathbb P((\mathbb C^4)^*)=3$, this cone has dimension
$3+3+3+1=10$. Next consider the Lie algebra flattening matrix
$
M_{\mathrm{Lie}}(f)\in\operatorname{Mat}_{20\times15}
$
representing
$\Psi_f:\mathfrak{sl}_4\to S^3\mathbb C^4$.
For a general point of $\mathcal V$, after a linear change of
coordinates we may take $f=x_1x_2x_3$. Its infinitesimal stabilizer in
$\mathfrak{sl}_4$ has dimension~$5$: it contains the three operators
$E_{14},E_{24},E_{34}$ together with the two-dimensional trace-zero
diagonal subspace
$$
\left\{
\operatorname{diag}(\lambda_1,\lambda_2,\lambda_3,0):
\lambda_1+\lambda_2+\lambda_3=0
\right\}.
$$
Hence
$
\rank M_{\mathrm{Lie}}(f)=15-5=10,
$
so all $11\times11$ minors of $M_{\mathrm{Lie}}(f)$ vanish on
$\mathcal V$.

The attention-specific rank constraint of
\Cref{prop:rank-bounds} gives a lower-degree equation. Consider
$$
N=
\begin{pmatrix}
y_{111} & y_{112} & y_{113} & y_{114} \\
y_{122} & y_{222} & y_{223} & y_{224} \\
y_{133} & y_{233} & y_{333} & y_{334} \\
y_{144} & y_{244} & y_{344} & y_{444}
\end{pmatrix}.
$$
Since $\rank(A)\leq1$, we have
$
\rank(N)\leq2\rank(A)+1\leq3.
$
Therefore
$
\det(N)=0
$
is a quartic equation for $\mathcal V$. A \texttt{Macaulay2} tangent-space computation at
$f=x_1x_2x_3$ shows that the locus
$
\{f:\rank M_{\mathrm{Lie}}(f)\leq 10\}
$
has codimension $10$ at this point. Since $\mathcal V$ is
irreducible of dimension $10$ and is contained in this locus,
$\mathcal V$ is the irreducible component through $f=x_1x_2x_3$.
The corresponding computation is available in the accompanying
repository.
A complete implicitization is computationally infeasible in this
example.
\end{example}

\begin{remark}
By \Cref{cor:rank-one-single-column}, the rank-one single-column
variety is the Chow variety of completely reducible cubic forms.
This variety is defined set-theoretically by the classical Brill
equations; see
\cite{briand2004brill,guan2018brill} and
\cite[Section~8.6]{landsberg2012tensors}.
The Lie algebra flattening and the matrix $N$ above provide equations
arising directly from the attention parametrization, but we do not
claim that these equations alone define the Chow~variety.
\end{remark}

\section{Cross-column monomials}\label{sec:cross-column}

The cross-column coefficients 
correspond to terms involving a context position $n$ and the target position $j$.
For a fixed pair $n\neq j$, these terms are linear in the target-column variables and quadratic in the context-column variables; 
grouping by target variables yields a tuple of quadratic forms in the context variables.
We first characterize the variety parametrized by these quadrics, and then derive 
pencil, flattening, Veronese-type, and resultant invariants as consequences of this structure. 

\begin{definition}[Slice quadrics]
\label{def:slice-quadrics}
Fix a target column $j\in[t]$ and a context column $n\in[t]\setminus\{j\}$. For each slice index $s\in[d]$, define the \emph{formal slice quadric}
\begin{equation}\label{eq:cross-column-quadrics}
q_s(y;z):=
\sum_{1\le p\le q\le d}
\omega(p,q)\,y_{(p,n),(q,n),(s,j)}\,z_pz_q,
\qquad
\omega(p,q)=
\begin{cases}
1,&p=q,\\
2,&p\neq q.
\end{cases}
\end{equation}
\end{definition}

\begin{definition}[Cross-column variety]\label{def:cross-column-variety}
Fix $n,j\in[t]$ with $n\neq j$. Fix $a \geq 1$.
The \emph{cross-column variety} $\mathcal C_a$ is the Zariski closure
$$
\mathcal C_a
=
\overline{
\left\{
\bigl(q_1(\mu(W);z),\ldots,q_d(\mu(W);z)\bigr)
:
W=(Q,K,V)
\right\}}
\subseteq
\bigl(\CC[z]_2\bigr)^d.
$$
\end{definition}

Thus, $\mathcal C_a$ records one representative cross-column coefficient block of the attention variety. For fixed $n\neq j$ and target index $s$, the coefficients of the cubic monomials $x_{pn}x_{qn}x_{sj}$ are packaged into the slice quadric $q_s(y;z)$ by replacing the context variables $x_{pn}$ with auxiliary variables $z_p$. By our scaling convention, the coefficient of $z_pz_q$ is exactly the coefficient of $x_{pn}x_{qn}x_{sj}$ in the attention output. Since the sequence-copy relations of \Cref{sec:linear} identify the blocks corresponding to different pairs $(n,j)$, it suffices to study the variety $\mathcal C_a$, whose structure is characterized below.

\begin{proposition}[Cross-column structure]
\label{prop:cross-column}
The cross-column variety admits the following structural description:
\begin{equation}\label{eq:cross-column-structure}
\mathcal C_a
=
\left\{
(\ell m_1,\ldots,\ell m_d):
\ell,m_1,\ldots,m_d\in\CC[z]_1,\ 
\dim\operatorname{Span}\{m_1,\ldots,m_d\}\leq a
\right\}.
\end{equation}
Thus, the slice quadrics have a common linear factor, and their remaining linear factors span a space of dimension at most $a$.
\end{proposition}

\begin{proof}
Let $v=(v_1,\ldots,v_d)^\top$ denote the unique row of $V$, and let
$a_s=(a_{1s},\ldots,a_{ds})^\top$ denote the $s$th column of $A$.
By \Cref{prop:sums-monomials} and the definition of the scaled coefficients, the coefficient of $z_p^2$ in $q_s(\mu(W);z)$ is $a_{ps}v_p$, while the coefficient of $z_pz_q$ for $p<q$ is $a_{qs}v_p+a_{ps}v_q$. These are exactly the coefficients of $(v^\top z)(a_s^\top z)$. Hence
$q_s(\mu(W);z)=(v^\top z)(a_s^\top z)$. Writing $\ell(z)=v^\top z$ and $m_s(z)=a_s^\top z$, we obtain
$q_s(\mu(W);z)=\ell m_s$ for every $s\in[d]$. Moreover, the linear forms
$m_s$ correspond to the columns of $A$, so
$\dim\operatorname{Span}\{m_1,\ldots,m_d\}=\rank(A)\leq a$.
This proves one inclusion in (\ref{eq:cross-column-structure}).

Conversely, suppose $(q_1,\ldots,q_d)$ is a tuple satisfying
$q_s=\ell m_s$ for $s\in[d]$, with $\dim\operatorname{Span}\{m_1,\ldots,m_d\}\leq a$. Write
$\ell(z)=v^\top z$ and $m_s(z)=a_s^\top z$, and let
$A=(a_1\ \cdots\ a_d)$. Then $\rank(A)\leq a$, so $A=K^\top Q$ for
some $K,Q\in\CC^{a\times d}$. Taking $V=v^\top$, the scaled
coefficient formulas give exactly the quadrics $q_1,\ldots,q_d$. Finally, the projectivization of the right-hand side of
(\ref{eq:cross-column-structure}) is the image of
$\PP^{d-1}\times
\{[A]\in\PP^{d^2-1}:\rank(A)\leq a\}$ under the morphism
$([\ell],[A])\mapsto[(\ell m_1,\ldots,\ell m_d)]$.
The domain is projective, so this image is Zariski closed. Its affine cone is therefore Zariski closed as well.
This proves~(\ref{eq:cross-column-structure}).
\end{proof}

\begin{corollary}[Reduction to the cross-column variety]
\label{cor:cross-column-reduction}
Assume $t\geq2$. After imposing the linear relations of
\Cref{sec:linear}, the attention variety is linearly isomorphic
to $\mathcal C_a$. In particular, the nonlinear single-column relations
are redundant.
\end{corollary}

\begin{proof}
The sequence-copy relations identify all cross-column blocks, while
\Cref{prop:symmetrization} expresses every single-column coefficient
linearly in terms of any representative cross-column block.
\end{proof}

We now study the defining equations of $\mathcal C_a$.

\subsection{Determinantal constraints}\label{subsec:determinantal-constraints}

We first record two direct consequences of (\ref{eq:cross-column-structure}).
The common linear factor gives rank-two pencil relations, while the bound
$\dim\operatorname{Span}\{m_1,\ldots,m_d\}\leq a$ gives a rank constraint
on the cross-column flattening.

For each $s\in[d]$, let $M_s(y)$ be the symmetric matrix representing
$q_s(y;z)$, so that $q_s(y;z)=z^\top M_s(y)z$.

\begin{proposition}[Rank-two pencil relations]
\label{prop:cross-column-pencil}
For every point of $\mathcal C_a$ and every
$\lambda=(\lambda_1,\ldots,\lambda_d)\in\CC^d$,
$$
\rank\left(\sum_{s=1}^d\lambda_s M_s\right)\leq 2.
$$
Consequently, the $3\times3$ minors of
$\sum_{s=1}^d\lambda_sM_s$ yield cubic equations for $\mathcal C_a$.
\end{proposition}

\begin{proof}
By \Cref{prop:cross-column}, we may write
$q_s=\ell m_s$, where $\ell(z)=v^\top z$ and $m_s(z)=a_s^\top z$. Hence
$$
M_s=\frac{1}{2}(va_s^\top+a_sv^\top).
$$

Any matrix $M$ in the span of $M_1,\ldots,M_d$ can be written as
$M=\sum_{s=1}^d\lambda_sM_s$ for some $\lambda_1,\ldots,\lambda_d\in\CC$.
By linearity,
$$
M
=
\frac{1}{2}\left(
v\left(\sum_{s=1}^d\lambda_sa_s\right)^\top
+
\left(\sum_{s=1}^d\lambda_sa_s\right)v^\top
\right).
$$
Setting $u=\sum_{s=1}^d\lambda_sa_s$, we obtain
$M=\frac{1}{2}(vu^\top+uv^\top)$. Since $M$ is the sum of two rank-one
matrices, its rank is at most $2$.

Therefore, all $3\times3$ minors of
$\sum_{s=1}^d\lambda_sM_s$ vanish. By multilinearity of the determinant,
each such minor is a homogeneous cubic polynomial in
$\lambda_1,\ldots,\lambda_d$. Since it vanishes for every~$\lambda$, each
of its coefficients vanishes, giving cubic equations in the
cross-column coefficients.
\end{proof}

We next use the bound
$\dim\operatorname{Span}\{m_1,\ldots,m_d\}\leq a$ in
(\ref{eq:cross-column-structure}). Define the cross-column flattening
$F(y)$ to be the $\binom{d+1}{2}\times d$ matrix whose rows are indexed
by $\{p,q\}\in\operatorname{Mult}_2([d])$ and whose columns are indexed
by $s\in[d]$, with
\begin{equation}\label{eq:cross-column-flattening}
F(y)_{\{p,q\},s}=y_{(p,n),(q,n),(s,j)}.
\end{equation}
Thus, the $s$th column of $F(y)$ is the vectorized upper-triangular
part of $M_s(y)$.

\begin{proposition}[Flattening rank constraints]
\label{prop:cross-column-flattening}
For every point of $\mathcal C_a$,
$
\rank F(y)\leq a.
$
Consequently, if $a<d$, all $(a+1)\times(a+1)$ minors of $F(y)$ vanish
on $\mathcal C_a$.
\end{proposition}

\begin{proof}
By \Cref{prop:cross-column}, write
$q_s=\ell m_s$, with $\ell(z)=v^\top z$ and $m_s(z)=a_s^\top z$,
and let $A=(a_1\ \cdots\ a_d)$. Since
$\dim\operatorname{Span}\{m_1,\ldots,m_d\}\leq a$, we have
$\rank(A)\leq a$, so we can express $A$ as the sum of at most $a$ outer products
for some $u^{(r)},w^{(r)}\in\CC^d$:
$$
A=\sum_{r=1}^a u^{(r)}(w^{(r)})^\top.
$$
Thus, the $s$th column of $A$ is
$
a_s=\sum_{r=1}^a w_s^{(r)}u^{(r)}.
$
Since $q_s(z)=(v^\top z)(a_s^\top z)$, the symmetric matrix representing
$q_s$ is
$$
M_s=\frac12(va_s^\top+a_sv^\top).
$$
Substituting the expansion of $a_s$ gives
$$
M_s
=
\sum_{r=1}^a w_s^{(r)}
\left[
\frac12\left(v(u^{(r)})^\top+u^{(r)}v^\top\right)
\right].
$$
For each $r\in[a]$, define
$$
S^{(r)}
=
\frac12\left(v(u^{(r)})^\top+u^{(r)}v^\top\right).
$$
This matrix is independent of $s$. 
By (\ref{eq:cross-column-flattening}), the $s$th column of $F(y)$ is
the vectorized upper-triangular part of $M_s$. Hence, for every
$\{p,q\}\in\operatorname{Mult}_2([d])$ and $s\in[d]$,
$$
F(y)_{\{p,q\},s}
=
(M_s)_{p,q}
=
\sum_{r=1}^a w_s^{(r)}S^{(r)}_{p,q}.
$$
Hence
$$
F(y)=\sum_{r=1}^a s^{(r)}(w^{(r)})^\top,
$$
where $s^{(r)}\in\CC^{\binom{d+1}{2}}$ is the vectorized
upper-triangular part of $S^{(r)}$. Since $F(y)$ is the sum of at most
$a$ rank-one matrices, $\rank F(y)\leq a$. Hence all
$(a+1)\times(a+1)$ minors vanish.
\end{proof}

\paragraph{The rank-one case}

When $a=1$, the two determinantal conditions above are also sufficient 
and therefore characterize the cross-column variety
set-theoretically.

\begin{theorem}
\label{thm:rank-one-cross-column}
A tuple $(q_1,\ldots,q_d)$ lies in $\mathcal C_1$ if and only if
$$
\rank F(y)\leq 1
\qquad\text{and}\qquad
\rank\left(\sum_{s=1}^d\lambda_sM_s(y)\right)\leq 2
\quad\text{for every }\lambda\in\CC^d.
$$
Equivalently, $\mathcal C_1$ is cut out set-theoretically by the
$2\times2$ minors of $F(y)$ together with the cubic pencil relations of
\Cref{prop:cross-column-pencil}.
\end{theorem}

\begin{proof}
The two conditions hold on $\mathcal C_1$ by
\Cref{prop:cross-column-pencil,prop:cross-column-flattening}.
Conversely, $\rank F(y)\leq1$ implies that
$q_s=c_sq$ for some quadratic form $q$ and scalars $c_s$. The pencil
condition gives $\rank(q)\leq2$. 
Over $\CC$, such a quadratic factors as
$q=\ell m$ for linear forms $\ell$ and $m$: after a linear change of coordinates, a nonzero
quadratic of rank at most $2$ is either $z_1^2$ or
$z_1^2+z_2^2=(z_1+iz_2)(z_1-iz_2)$.
Hence
$q_s=\ell(c_sm)$ for every $s$, and
(\ref{eq:cross-column-structure}) gives
$(q_1,\ldots,q_d)\in\mathcal C_1$.
\end{proof}

\begin{example}\label{ex:3121}
Fix $d=3$, $d'=1$, $t=2$, and $a=1$, and take the target and context
columns to be $j=1$ and $n=2$. For brevity, write
$$
y_{pqs}:=y_{(p,2),(q,2),(s,1)},
\qquad 1\leq p\leq q\leq3,\quad s\in[3].
$$
Then
$$
M_s=
\begin{pmatrix}
y_{11s} & y_{12s} & y_{13s}\\
y_{12s} & y_{22s} & y_{23s}\\
y_{13s} & y_{23s} & y_{33s}
\end{pmatrix},
\qquad
F(y)=
\begin{pmatrix}
y_{111} & y_{112} & y_{113}\\
y_{121} & y_{122} & y_{123}\\
y_{131} & y_{132} & y_{133}\\
y_{221} & y_{222} & y_{223}\\
y_{231} & y_{232} & y_{233}\\
y_{331} & y_{332} & y_{333}
\end{pmatrix}.
$$

By \Cref{thm:rank-one-cross-column}, the $2\times2$ minors of $F(y)$
together with the cubic pencil relations cut out the cross-column
variety $\mathcal C_1$ set-theoretically. Since $F(y)$ is $6\times3$,
\Cref{prop:cross-column-flattening}~gives
$
\binom{6}{2}\binom{3}{2}=45
$
quadratic equations.

For the pencil relations, there is a single $3\times3$ minor, namely
$$
\det(\lambda_1M_1+\lambda_2M_2+\lambda_3M_3).
$$
It vanishes identically in the three auxiliary variables
$\lambda_1,\lambda_2,\lambda_3$, one for each of the three slices.
Its determinant is a homogeneous cubic in these variables, so
equating its
$
\binom{3+3-1}{3}=10
$
coefficients to zero gives $10$ cubic equations.

A computation in \texttt{Macaulay2} shows that, together with the
$10$ linear relations from \Cref{sec:linear}, these $45$ quadrics and
$10$ cubics form a minimal generating set for the defining ideal $J$.
The corresponding computation is available in the accompanying
repository.
\end{example}

\subsection{Cross-column determinantal syzygies}

The flattening constraints of \Cref{prop:cross-column-flattening} arise
from the dimension bound $\dim\operatorname{Span}\{m_1,\ldots,m_d\}\leq a$.
When $a\geq d$, this bound imposes no rank restriction on the
$d$ columns of $F(y)$. However, the common linear factor in
(\ref{eq:cross-column-structure}) gives the flattening a special
factorization, which forces linear relations among its maximal~minors. Unlike the pencil constraints of \Cref{prop:cross-column-pencil}, which
bound the rank of linear combinations of the slice matrices $M_s$,
these relations arise from the factorization of the full flattening.

\begin{theorem}[Cross-column determinantal syzygies]
\label{thm:context_syzygies}
Let $d\geq2$ and assume $a\geq d$. Writing
$D=\binom{d+1}{2}$, there exist at least
$
\binom{D}{d}-\binom{2d-1}{d}
$
linearly independent degree-$d$ equations for $\mathcal C_a$, arising
as linear combinations of the maximal $d\times d$ minors of $F(y)$.
\end{theorem}

\begin{proof}
Write $v=(v_1,\ldots,v_d)^\top$. By \Cref{prop:sums-monomials},
evaluating the flattening on the parametrization gives
$$
F(\mu(W))[\{k_1,k_2\},l]
=
\frac{1}{2}
\bigl(a_{k_2l}v_{k_1}+a_{k_1l}v_{k_2}\bigr).
$$
The right-hand side is linear in the $l$th column of $A$, with
coefficients depending only on $v$. Define $\Phi(v)$ to be the
$D\times d$ matrix whose row indexed by $\{k_1,k_2\}$ has, when
$k_1\neq k_2$, zeros everywhere except at column $k_1$, where it has
entry $v_{k_2}/2$, and at column $k_2$, where it has entry
$v_{k_1}/2$. If $k_1=k_2$, the corresponding row has the single
nonzero entry $v_{k_1}$ at column~$k_1$. Then
$$
(\Phi(v)A)[\{k_1,k_2\},l]
=
\frac{1}{2}
\bigl(a_{k_2l}v_{k_1}+a_{k_1l}v_{k_2}\bigr)
=
F(\mu(W))[\{k_1,k_2\},l],
$$
where for $k_1=k_2$ the expression reduces to
$a_{k_1l}v_{k_1}$. Hence
$$
F(\mu(W))=\Phi(v)A.
$$

Any maximal $d\times d$ minor of $F(\mu(W))$ formed by selecting $d$
rows $R\subset[D]$ therefore evaluates to
$$
\det(F_R(\mu(W)))=\det(\Phi_R(v))\det(A).
$$
Since the entries of $\Phi(v)$ are linear in $d$ variables,
$\det(\Phi_R(v))$ is a homogeneous polynomial of degree $d$ in those
variables. The vector space of homogeneous polynomials of degree $d$
in $d$ variables has dimension
$
\binom{2d-1}{d}.
$

The matrix $F(y)$ has $\binom{D}{d}$ maximal minors. Since their
evaluations lie, after factoring out the common factor $\det(A)$, in a
vector space of dimension at most $\binom{2d-1}{d}$, there exist at
least
$
\binom{D}{d}-\binom{2d-1}{d}
$
linearly independent linear relations among them. Since the maximal
minors of the formal matrix $F(y)$ are linearly independent
polynomials, the corresponding linear combinations give the same
number of linearly independent equations vanishing on $\mathcal C_a$. Since each maximal minor of $F(y)$ has degree $d$ in
the coordinates $y$, these give degree-$d$ equations for
$\mathcal C_a$.
\end{proof}

\begin{example}
Consider $d=3$, $d'=1$, $t=2$, and $a=3$, where, since $a\geq d$, we have no rank constraint on the attention parameter matrix $A$. 
In this case, the matrix $F(y)$ has dimensions $6\times3$ and 
hence it has $\binom{6}{3}=20$ maximal $3\times3$ minors. 
On the other hand, by the proof of \Cref{thm:context_syzygies}, these minors factor through the matrix $\Phi(v)$, so they span a subspace of cubic forms in $v_1,v_2,v_3$ of dimension at most
$$
\dim \Sym^3(\CC^3)^*=\binom{2\cdot3-1}{3}=10.
$$
A computation shows that this span has dimension $10$. Hence, there are exactly $20-10=10$ linear relations among the maximal minors of $F(y)$, producing $10$ cubic invariants. These relations can be computed independently of elimination. Namely, one first forms the $20$ maximal minors of the
matrix $\Phi(v)$ and expresses them in the standard monomial basis
$$
v_1^3,\ v_1^2v_2,\ v_1^2v_3,\ v_1v_2^2,\ v_1v_2v_3,\ v_1v_3^2,\ v_2^3,\ v_2^2v_3,\ v_2v_3^2,\ v_3^3.
$$
Expanding each of the $20$ maximal minors of $\Phi(v)$ in the standard cubic monomial basis of $\Sym^3(\CC^3)^*$ produces a $10\times20$ coefficient matrix: its columns are indexed by the $20$ minors, its rows by the $10$ cubic monomials in $v_1,v_2,v_3$, and each entry is the coefficient of the corresponding monomial in the corresponding minor. Its kernel records the linear dependencies among the $20$ minors. One may compute this directly in \texttt{Macaulay2} by listing all row triples, forming all maximal minors, building the coefficient matrix in the cubic monomial basis, and taking \texttt{gens kernel} of that matrix. Applying the resulting kernel vectors to the corresponding $20$ maximal minors of $F(y)$ produces the cubic invariants. These $10$ cubics are distinct from the $10$ cubics derived in \Cref{prop:cross-column-pencil}, as verified by computation; see accompanying repository. 
\end{example}

\subsection{Veronese-type invariants}

The structural description (\ref{eq:cross-column-structure}) also gives
rise to higher-degree relations among minors of the cross-column
coefficients. We begin with an example that motivates the resulting
Veronese-type invariants. 

\begin{example}
\label{ex:quartic} 
Let $d=t=2$, $a=2$, and $d^\prime=1$. 
We restrict our focus to the first output token ($j=1$). We examine the cross-column monomials that mix the target column $j=1$ with the context column $n=2$. These monomials are of degree two in the context column and degree one in the target column. 

Let $y_{(k_1,2),(k_2,2),(k_3,1)}$ denote the 
scaled coefficient of the monomial $x_{k_1 2} x_{k_2 2} x_{k_3 1}$. 
Under our parametrization map $\mu$, evaluating these coordinates using \Cref{prop:sums-monomials} yields:
\begin{align*}
y_{(1,2),(1,2),(1,1)}(W) &= a_{11}v_1,
&
y_{(1,2),(1,2),(2,1)}(W) &= a_{12}v_1,
\\
y_{(1,2),(2,2),(1,1)}(W) &= \tfrac12(a_{21}v_1+a_{11}v_2),
&
y_{(1,2),(2,2),(2,1)}(W) &= \tfrac12(a_{22}v_1+a_{12}v_2),
\\
y_{(2,2),(2,2),(1,1)}(W) &= a_{21}v_2,
&
y_{(2,2),(2,2),(2,1)}(W) &= a_{22}v_2.
\end{align*}

Eliminating the parameters to find the implicit algebraic relations reveals a minimal quartic invariant in the 
variables $y$. 
This quartic has a determinantal structure of specific $2 \times 2$ determinants: 
\begin{align}
\begin{split} 
\label{eq:quartic0}
   q(y) \!&=\!  4\!\cdot\!\underbrace{\det\begin{pmatrix}
        y_{(1,2),(1,2),(2,1)} & y_{(1,2),(1,2),(1,1)}\\
        y_{(1,2),(2,2),(2,1)} & y_{(1,2),(2,2),(1,1)}
    \end{pmatrix}}_{d_1(y)} \!\cdot\! \underbrace{\det\begin{pmatrix}
        y_{(1,2),(2,2),(2,1)} & y_{(1,2),(2,2),(1,1)}\\
        y_{(2,2),(2,2),(2,1)} & y_{(2,2),(2,2),(1,1)} 
    \end{pmatrix}}_{d_2(y)} 
    \\
    &\quad - \Bigg(\underbrace{\det\begin{pmatrix}
        y_{(1,2),(1,2),(2,1)} & y_{(1,2),(1,2),(1,1)}\\
        y_{(2,2),(2,2),(2,1)} & y_{(2,2),(2,2),(1,1)} 
    \end{pmatrix}}_{d_3(y)}\Bigg)^2.   
    \end{split} 
\end{align}

To understand why this 
quartic $q(y)$ vanishes on our attention variety, we substitute the parametrization and evaluate $d_1, d_2,$ and $d_3$ explicitly: 
%
\begin{align*}
    d_1(\mu(W)) &= a_{12}v_{1} \cdot \tfrac{1}{2}(a_{21}v_{1} + a_{11}v_{2}) - a_{11}v_{1} \cdot \tfrac{1}{2}(a_{22}v_{1} + a_{12}v_{2}) = -\tfrac{1}{2}v_{1}^2\det(A),\\
    d_2(\mu(W)) &= \tfrac{1}{2}(a_{22}v_{1} + a_{12}v_{2})(a_{21}v_{2}) - \tfrac{1}{2}(a_{21}v_{1} + a_{11}v_{2})(a_{22}v_{2}) = -\tfrac{1}{2}v_{2}^2\det(A),\\
    d_3(\mu(W)) &= a_{12}v_{1}(a_{21}v_{2}) - a_{11}v_{1}(a_{22}v_{2}) = -v_{1}v_{2}\det(A).
\end{align*}

The right hand sides show that 
the map defining these determinants factors 
through the Veronese embedding of $\mathbb{P}^1 \hookrightarrow \mathbb{P}^2$, 
recognized in statistics as the \textit{Hardy-Weinberg curve}: 
$$
\psi: W \mapsto (d_1, d_2, d_3) = -\tfrac{1}{2}\det(A) \cdot (\underbrace{\quad v_{1}^2\quad }_{z_1}, \underbrace{\quad v_{2}^2\quad }_{z_2}, \underbrace{\quad 2v_{1}v_{2}\quad }_{z_3}).
$$
The defining equation of this curve is $4z_1z_2 - z_3^2 = 0$. Translating this back to our determinants exactly yields the condition $4d_1d_2 - d_3^2 = 0$, showing that the 
quartic $q(y)$ vanishes identically on the attention variety. A 2-dimensional slice of this quartic is depicted in \Cref{fig:quartic}.  
\end{example}

The determinantal structure observed in \Cref{ex:quartic} 
extends to arbitrary sequence length $t\geq2$ and gives rise to a general family of Veronese-type invariants.

\paragraph{Veronese embeddings and their invariants.}  
We briefly recall the Veronese embedding and its defining equations.
For $u = (u_1, \dots, u_n) \in \CC^n$, the $d$th Veronese map 
sends $u$ to the collection of all monomials of degree $d$ in $u_1, \dots, u_n$. 
Projectively, this gives the embedding 
$\nu_d\colon \mathbb{P}^{n-1} \to \mathbb{P}^{N-1}$, where $N = \binom{n+d-1}{d}$. 
For example, for $n=2$ and $d=2$, 
$(u_1, u_2) \mapsto (u_1^2, u_1 u_2, u_2^2)$. 
Let $z_\alpha$ denote the coordinate corresponding to the monomial $u^\alpha = u_1^{\alpha_1} \cdots u_n^{\alpha_n}$, where $\sum \alpha_i = d$. 
These coordinates satisfy quadratic binomial relations:
$$
z_\alpha z_\beta - z_\gamma z_\delta = 0, 
$$
whenever $\alpha + \beta = \gamma + \delta$. 
For instance, when $n=d=2$, this gives $(u_1^2)(u_2^2) = (u_1 u_2)^2$. 
Such quadratic binomials generate the defining ideal of the Veronese variety; see, for example, \cite[Chapter~4]{landsberg2012tensors}. 

In computational settings, these quadratic relations are frequently expressed through determinantal flattenings. 
For $1 \le k < d$, we can arrange the coordinates $z_\gamma$ into a symmetric matrix $M^{(k)}$, known as the \textit{catalecticant matrix}~\cite{pucci1998veronese}, whose rows and columns are indexed by monomials of degrees $k$ and $d-k$, respectively, with entries 
$M^{(k)}_{\alpha, \beta} = z_{\alpha+\beta}$. 
The defining ideal of the Veronese variety is generated by the $2\times 2$ minors of $M^{(k)}$ for any fixed choice of $k$ \cite[Section 4]{LORS2019}.
\exampleqed

\medskip 

By \Cref{prop:cross-column}, on $\mathcal C_a$ the slice quadrics factor as 
$q_s=(v^\top z)(a_s^\top z)$. 
Suitable $r\times r$ minors of their coefficient matrices
isolate degree-$r$ monomials in the coordinates of the common factor $v^\top z$, multiplied by a 
common minor of $A$. 
As the monomial varies, these minors therefore behave like Veronese coordinates and satisfy the quadratic relations of the $r$th Veronese embedding, 
yielding polynomial equations for $\mathcal C_a$. 

\begin{theorem}[Veronese-type invariants]
\label{thm:veronese}
Let $t > 1$, and fix a target column $j \in [t]$ and a context column $n \in [t]\setminus\{j\}$. 
Fix $2 \le r \le \min(a, d)$ and ordered tuples $K = (k_1, \dots, k_r)$ 
and $L = (l_1, \dots, l_r)$ 
of $r$ distinct 
indices in $[d]$. 
For every multi-index $\alpha = (\alpha_1, \dots, \alpha_r) \in \mathbb{N}^r$ 
with $\sum \alpha_m = r$, 
there exists an $r \times r$ matrix $M_\alpha(y)$ built from the formal image coordinates $y$ such that, 
writing $D_\alpha(y) = \det(M_\alpha(y))$, 
we have 
$$
D_\alpha(\mu(W)) = \left( \prod_{m=1}^r v_{ k_m}^{\alpha_m} \right) \det(A_{K, L}),
$$
where $A_{K,L}$ is the 
submatrix of $A$ with rows indexed by $K$ and columns indexed by $L$. 

Thus, 
as $\alpha$ varies, the $D_\alpha(y)$ evaluate to the 
degree-$r$ Veronese monomials, scaled by the common factor $\det(A_{K, L})$. 
Arranging these 
into a catalecticant matrix yields a set of $2 \times 2$ minors 
$D_\alpha D_\beta - D_\gamma D_\delta$ 
that vanish on $\mathcal C_a$. 
These minors provide a family of polynomial equations of degree $2r$ in the coefficient ring $\CC[y]$ for $\mathcal C_a$. 
\end{theorem}

\begin{proof}
We construct the matrix $M_\alpha(y)$ by choosing a map $f:[r]\to[r]$ such that $|f^{-1}(m)|=\alpha_m$ for every~$m$, and such that the directed graph of $f$ has no cycles of length greater than $1$. 
%
To construct such an $f$, assign a self-loop $f(m)=m$ for each $m$ with $\alpha_m\ge 1$. This uses one incoming edge at each such vertex and leaves
$$
\sum_{\alpha_m\ge 1}(\alpha_m-1)
=
r-\#\{m:\alpha_m\ge 1\}
$$
incoming edges still to be assigned. 
The remaining vertices are precisely those with $\alpha_p=0$, and there are exactly
$$
r-\#\{m:\alpha_m\ge 1\}
$$
of them. Assign each such vertex $p$ to one of the remaining required targets so as to realize the multiplicities $\alpha_m$. Since vertices with $\alpha_p=0$ receive no incoming edges, they cannot lie on a directed cycle. Hence every cycle of $f$ is a self-loop.

We assign the $p$th row of $M_\alpha(y)$ the context pair $\{k_{f(p)}, k_p\}$, and the $q$th column the target index $l_q$. 

The matrix entry is 
defined as: 
$$
(M_\alpha)_{p,q} = \omega(k_{f(p)}, k_p) \cdot y_{(k_{f(p)}, n), (k_p, n), (l_q, j)}.
$$
The scaling factor $\omega$ is defined in (\ref{eq:cross-column-quadrics}); the role of 
$\omega(k_{f(p)},k_p)$ is to clear the symmetry normalization in the cross-column coordinates, so that the determinants recover the standard Veronese monomials without additional binomial coefficients. 

By \Cref{prop:sums-monomials}, evaluated on the parametrization $\mu$, this unscaled entry expands to $a_{k_p l_q} v_{ k_{f(p)}} + a_{k_{f(p)} l_q} v_{ k_p}$ when $f(p) \neq p$, and consists of the single term $a_{k_p l_q} v_{ k_p}$ when $f(p) = p$. Let $s$ be the number of rows where $f(p) \neq p$. By the multilinearity of the determinant with respect to these $s$ rows, $\det(M_\alpha(\mu(W)))$ expands into a sum of $2^s$ terms. Each term is indexed by a subset
$I\subseteq\{p:f(p)\neq p\}$, specifying the rows in which the second
summand is chosen; in all remaining rows, the first summand is chosen.

The ``primary'' term in this expansion selects $a_{k_p l_q} v_{ k_{f(p)}}$ for every row $p$, including the self-loop rows, where this is the single available term. This allows the matrix to factor exactly as $\diag(v_{ k_{f(1)}}, \dots, v_{ k_{f(r)}}) A_{K, L}$. Since $|f^{-1}(m)|=\alpha_m$ for every $m$, we have
$
\prod_{p=1}^r v_{k_{f(p)}}
=
\prod_{m=1}^r v_{k_m}^{\alpha_m}.
$ 
Its determinant is $(\prod_{m=1}^r v_{ k_m}^{\alpha_m}) \det(A_{K, L})$, which provides the desired Veronese coordinate.

Any other term in the determinant expansion must select the ``secondary'' part, $a_{k_{f(p)} l_q} v_{ k_p}$, for some non-empty subset of rows $I \subseteq \{p \mid f(p) \neq p\}$. The rows of the corresponding $A$ submatrix are then indexed by $k_{f(p)}$ for $p \in I$ and by $k_p$ for $p \notin I$. For this determinant to be non-zero, these row indices must again be exactly the set $K$, which is equivalent to requiring $f(I)=I$. Since $I$ is finite, this means that the restriction $f|_I$ is a permutation of $I$, and hence contains a cycle. Because every $p \in I$ satisfies $f(p)\neq p$, such a cycle must have length at least $2$, contradicting the construction of $f$. Therefore the row indices cannot all be distinct, so the determinant vanishes.

Therefore, $D_\alpha(\mu(W))$ evaluates exactly to the single primary term. Since these parameterized determinants map perfectly to the Veronese coordinates, the $2 \times 2$ minors of their corresponding formal catalecticant matrix vanish on the image of $\mu$, and hence on $\mathcal C_a$.
\end{proof}

\begin{example}[Example \ref{ex:quartic} revisited]
We can now re-examine the quartic invariant from Example \ref{ex:quartic} through the lens of \Cref{thm:veronese}. 
In that example, we looked for invariants of degree $2r$, where $r=2$. 
We selected the context indices $K=(1,2)$ and target indices $L=(2,1)$. The degree $r=2$ multi-indices are $\alpha \in \{(2,0), (1,1), (0,2)\}$. Following the construction in \Cref{thm:veronese}, we build the three $2 \times 2$ matrices $M_\alpha(y)$. Recall that the entries are unscaled by the factor $\omega(u,v)$, defined in (\ref{eq:cross-column-quadrics}), to clear the fractional denominators from the formal coordinates. 
For $\alpha = (2,0)$, the mapping $f$ defined in the proof of \Cref{thm:veronese}, forces $f(1)=1$ and $f(2)=1$. 
The resulting matrix scales the second row by $\omega(1,2)=2$: 
$$ 
M_{(2,0)}(y) = \begin{pmatrix} 
    y_{(1,2),(1,2),(2,1)} & y_{(1,2),(1,2),(1,1)}\\
    2y_{(1,2),(2,2),(2,1)} & 2y_{(1,2),(2,2),(1,1)}
\end{pmatrix}. 
$$
Thus, $D_{(2,0)}(y) = \det(M_{(2,0)}(y)) = 2 d_1(y)$. For $\alpha = (0,2)$, the mapping $f$ forces $f(1)=2$ and $f(2)=2$. Here, the first row is scaled by $\omega(2,1)=2$: 
$$ 
M_{(0,2)}(y) = \begin{pmatrix} 
    2y_{(1,2),(2,2),(2,1)} & 2y_{(1,2),(2,2),(1,1)}\\
    y_{(2,2),(2,2),(2,1)} & y_{(2,2),(2,2),(1,1)}
\end{pmatrix}. 
$$ 
This gives $D_{(0,2)}(y) = \det(M_{(0,2)}(y)) = 2 d_2(y)$. Finally, for $\alpha = (1,1)$, the mapping $f$ is the identity ($f(1)=1, f(2)=2$). Both rows use matching context indices, so $\omega=1$:
$$ 
M_{(1,1)}(y) = \begin{pmatrix} 
    y_{(1,2),(1,2),(2,1)} & y_{(1,2),(1,2),(1,1)}\\
    y_{(2,2),(2,2),(2,1)} & y_{(2,2),(2,2),(1,1)}
\end{pmatrix}, 
$$
yielding $D_{(1,1)}(y) = \det(M_{(1,1)}(y)) = d_3(y)$.

By \Cref{thm:veronese}, these determinants $D_\alpha(y)$ represent the
degree-$2$ Veronese coordinates. We therefore form the $(1,1)$
catalecticant matrix, whose rows and columns are indexed by the
degree-$1$ multi-indices
$\beta,\gamma\in\{(1,0),(0,1)\}$.
The entry in row $\beta$ and column $\gamma$ is
$D_{\beta+\gamma}(y)$:

$$ 
\begin{pmatrix}
    D_{(2,0)}(y) & D_{(1,1)}(y) \\
    D_{(1,1)}(y) & D_{(0,2)}(y)
\end{pmatrix} = \begin{pmatrix}
    2 d_1(y) & d_3(y) \\
    d_3(y) & 2 d_2(y)
\end{pmatrix}. 
$$
The defining ideal of this Veronese embedding is generated by the determinant of this catalecticant matrix. Computing this determinant exactly recovers our formal quartic:
$$ 
\det \begin{pmatrix} 2 d_1(y) & d_3(y) \\ d_3(y) & 2 d_2(y) \end{pmatrix} = 4d_1(y)d_2(y) - (d_3(y))^2 = q(y). 
$$
The theorem thus explains both the determinantal structure of the minors $d_k$ and the presence of the factor of $4$ in (\ref{eq:quartic0}) as natural consequences of unscaling the binomial coefficients in the Veronese map. 
\end{example}

\begin{corollary}[Cross-target invariants]
\label{cor:cross_relations}
Let the assumptions of \Cref{thm:veronese} hold. 
Fix a choice of $r$ context indices $K$, and let $L$ and $L^\prime$ be two distinct choices of $r$ target indices. 
For any multi-index $\alpha$ of degree~$r$, let $D_{\alpha, L}(y)$ and $D_{\alpha, L^\prime}(y)$ denote the corresponding degree-$r$ Veronese coordinates. 
Define the $\binom{2r-1}{r} \times 2$~matrix 
$$ 
M_{L, L^\prime}(y) = \begin{pmatrix} \vdots & \vdots \\ D_{\alpha, L}(y) & D_{\alpha, L^\prime}(y) \\ \vdots & \vdots \end{pmatrix}, 
$$
where the rows are indexed by all multi-indices $\alpha$ of degree $r$. All $2 \times 2$ minors of $M_{L, L^\prime}(y)$ are 
polynomials of degree $2r$ that vanish on $\mathcal C_a$. 
\end{corollary}

\begin{proof}
By \Cref{thm:veronese}, evaluating the formal coordinates on the parametrization map $\mu$ yields:
$$ 
D_{\alpha, L}(\mu(W)) = v^\alpha \det(A_{K,L}) \quad \text{and} \quad D_{\alpha, L^\prime}(\mu(W)) = v^\alpha \det(A_{K,L^\prime}), 
$$
where $v^\alpha = \prod_{m=1}^r v_{ k_m}^{\alpha_m}$. The two columns of the evaluated matrix $M_{L, L^\prime}(\mu(W))$ share the exact same Veronese parameters $v^\alpha$ and differ only by the scalar 
minors of $A$. 
Because the columns are strictly proportional, the matrix has rank at most $1$ on the image of $\mu$. Thus, its $2 \times 2$ minors vanish on the image of $\mu$, and hence on $\mathcal C_a$. Since each entry $D_\alpha(y)$ is a homogeneous polynomial of degree $r$ in the formal variables $y$, taking the $2 \times 2$ minors generates invariant polynomials of degree $2r$. 
\end{proof}

This is not a special case of the catalecticant construction in
Theorem \ref{thm:veronese}; rather, it follows from the same parametrization formula
by comparing two different choices of target indices $L$ and $L'$.

\begin{corollary}[Block Veronese invariants]
\label{cor:block_veronese}
Let the assumptions of \Cref{thm:veronese} hold. 
Fix a choice of $r$ context indices $K$. 
Let $C_K(y)$ be the matrix with rows  indexed by
multi-indices $\beta$ of degree $1$ in $r$ variables, 
and columns indexed by pairs $(\gamma,L)$, where $\gamma$ is a multi-index of degree $r-1$
and $L$ is a choice of $r$ target indices, whose entry in row $\beta$ and column $(\gamma,L)$ is 
$$
C_K(y)_{\beta,(\gamma,L)} = D_{\beta+\gamma,L}(y). 
$$
Then all $2\times 2$ minors of $C_K(y)$ are 
polynomials of degree
$2r$ that vanish on $\mathcal C_a$. 
\end{corollary}

\begin{proof}
By \Cref{thm:veronese},
$$
D_{\beta+\gamma,L}(\mu(W))
= v^{\beta+\gamma}\det(A_{K,L})
= v^\beta \bigl(v^\gamma \det(A_{K,L})\bigr).
$$
Hence the evaluated matrix $C_K(\mu(W))$ factors as an outer product
$$
C_K(\mu(W))
=
\bigl(v^\beta\bigr)_\beta
\bigl(v^\gamma \det(A_{K,L})\bigr)_{(\gamma,L)}^{\!\top},
$$
so it has rank at most $1$. Therefore all of its $2\times 2$ minors vanish on
the image of $\mu$, and hence on $\mathcal C_a$. Since each entry $D_{\beta+\gamma,L}(y)$ is homogeneous of
degree $r$ in the formal variables $y$, these minors are homogeneous invariant
polynomials of degree $2r$. 
\end{proof}

For each fixed $L$, the corresponding block of $C_K(y)$ is the
$1+(r-1)$ catalecticant from \Cref{thm:veronese}; allowing $L$ to vary
also produces mixed-target minors.

\begin{remark}
The same construction applies to every decomposition $r = s + (r - s)$ with $1 \leq s \leq r - 1$, by indexing the rows and columns by multi-indices of degrees $s$ and $r-s$, respectively. We give the
construction explicitly only for $s=1$, since this is the case used in
the computations and examples below; the other cases are obtained
analogously.
\end{remark}

\begin{remark}
\Cref{cor:cross_relations,cor:block_veronese} capture
different rank-one structures. 
\Cref{cor:cross_relations} compares two target choices $L$ and $L'$ using the full Veronese coordinate vectors $(D_{\alpha,L})_{\alpha}$ and $(D_{\alpha,L'})_{\alpha}$, so its $2\times 2$ 
minors involve two fixed multi-indices $\alpha,\alpha'$ and two target sets. 
By contrast, \Cref{cor:block_veronese} first arranges the coordinates 
according to decompositions $\alpha=\beta+\gamma$ with $|\beta|=1$ and
$|\gamma|=r-1$, and then takes $2\times 2$ minors in the larger block matrix. 
This recovers the catalecticant minors for each fixed $L$, and also produces
mixed minors involving different pairs $(\gamma,L)$ and $(\gamma',L')$, such as
$$
\det\begin{pmatrix}
D_{\beta+\gamma,L}(y) & D_{\beta+\gamma',L'}(y)\\
D_{\beta'+\gamma,L}(y) & D_{\beta'+\gamma',L'}(y)
\end{pmatrix}. 
$$
These mixed minors do not in general appear in \Cref{cor:cross_relations}. 
Conversely, \Cref{cor:cross_relations}
allows arbitrary pairs $\alpha,\alpha'$ across two target sets, and such pairs
need not share a common decomposition with a fixed $\gamma$, so those minors do
not in general arise from \Cref{cor:block_veronese}. 
Thus the two families overlap, but neither contains the other in general. 
\end{remark}

\begin{example}\label{ex:d3_invariants}
Let $d=3$, $d'=1$, $t=2$, and $a=3$, and fix a target column $j\in[t]$. 
In this case, the defining ideal $J$ is minimally generated by $10$ linear relations from \Cref{sec:linear}, $10$ cubics from \Cref{prop:cross-column-pencil}, $10$ cubics from \Cref{thm:context_syzygies}, and $90$ quartics. 
The quartics are explained as follows. 
The Veronese-type constructions from \Cref{thm:veronese} and Corollaries~\ref{cor:cross_relations} and~\ref{cor:block_veronese} together produce $45$ quartics, but only $18$ of these remain in a minimal generating set after removing redundancies.  
Overall, $9$ minimal quartics come from \Cref{thm:veronese}, $9$ more come from \Cref{cor:block_veronese}, and the remaining $72$ quartics come from the resultants discussed in \Cref{sec:resultants}. 

The Veronese construction applies for both $r=2$ and $r=3$. 
For $r=2$, it produces quartic invariants. For $r=3$, we necessarily take $K = L=(1,2,3)$, 
so the determinants $D_\alpha(y)$, indexed by the ten multi-indices $\alpha\in\mathbb N^3$ with $|\alpha|=3$, behave like the coordinates of the degree-$3$ Veronese embedding of $\mathbb P^2$. These sextic relations are not needed for a minimal generating set of $J$, but they illustrate the higher-degree constraints provided by \Cref{thm:veronese}. For example, for $\alpha=(1,2,0)$ we may choose $f:[3]\to[3]$ given by
$$
f(1)=1,\qquad f(2)=2,\qquad f(3)=2.
$$
Then for each context column $n\neq j$ the matrix $M_{(1,2,0)}(y)$ is
$$
M_{(1,2,0)}(y)=
\begin{pmatrix}
y_{(1,n),(1,n),(1,j)} & y_{(1,n),(1,n),(2,j)} & y_{(1,n),(1,n),(3,j)}\\
y_{(2,n),(2,n),(1,j)} & y_{(2,n),(2,n),(2,j)} & y_{(2,n),(2,n),(3,j)}\\
2y_{(2,n),(3,n),(1,j)} & 2y_{(2,n),(3,n),(2,j)} & 2y_{(2,n),(3,n),(3,j)}
\end{pmatrix},
$$
and
$$
D_{(1,2,0)}(\mu(W))=v_{1}v_{2}^2\det(A_{K,L}).
$$

Arranging the ten determinants $D_\alpha(y)$ with $|\alpha|=3$ into the catalecticant matrix gives
$$
C_3(y)=
\begin{pmatrix}
D_{(3,0,0)} & D_{(2,1,0)} & D_{(2,0,1)} & D_{(1,2,0)} & D_{(1,1,1)} & D_{(1,0,2)}\\
D_{(2,1,0)} & D_{(1,2,0)} & D_{(1,1,1)} & D_{(0,3,0)} & D_{(0,2,1)} & D_{(0,1,2)}\\
D_{(2,0,1)} & D_{(1,1,1)} & D_{(1,0,2)} & D_{(0,2,1)} & D_{(0,1,2)} & D_{(0,0,3)}
\end{pmatrix}.
$$
Its $2\times 2$ minors vanish on the attention variety. For instance,
$$
D_{(3,0,0)}(y)D_{(1,2,0)}(y)-D_{(2,1,0)}(y)^2=0.
$$
Since each $D_\alpha(y)$ is cubic in the coefficient variables, these are sextic invariants.
\end{example}

\subsection{Quartics from resultants}\label{sec:resultants}

The previous subsection produced quartic and higher-degree invariants from rank-one catalecticant matrices. We now describe a different source of quartics, based on Sylvester resultants. By \Cref{prop:cross-column}, the slice quadrics $q_s(y;z)$ of \Cref{def:slice-quadrics} share a common linear factor on $\mathcal C_a$. Restricting them to a line in $\PP^{d-1}$ turns them into binary quadrics, so their Sylvester resultants vanish. Further quartics arise by taking resultants of linear combinations of the restricted slice quadrics.

We impose this line restriction in order to reduce to binary quadrics, for which the Sylvester resultant has a concrete determinantal expression. Without this restriction, one is led instead to multivariate elimination conditions for quadrics in $d$ variables, which are substantially less tractable computationally.

\paragraph{Resultants.} The \textit{resultant} is a classical elimination polynomial that detects when a system of homogeneous polynomials has a common nontrivial zero. 
In this subsection, we use the classical resultant of two binary forms, that is, two homogeneous polynomials in two variables.
For two binary forms $f$ and $g$, the resultant $\Res(f,g)$ is a polynomial in the coefficients of $f$ and $g$ that vanishes if and only if $f$ and $g$ have a common factor, equivalently, a common zero in $\PP^1$. 
In the special case of binary quadrics, $\Res(f,g)$ is given explicitly by the determinant of the Sylvester matrix. 

\begin{definition}[Sylvester resultant of two binary quadrics]\label{def:sylvester}
Let
$$
f(\lambda_1,\lambda_2)=A_1\lambda_1^2+B_1\lambda_1\lambda_2+C_1\lambda_2^2
\quad\text{and}\quad
g(\lambda_1,\lambda_2)=A_2\lambda_1^2+B_2\lambda_1\lambda_2+C_2\lambda_2^2
$$
be binary quadrics. Their \emph{Sylvester resultant} is
$$
\operatorname{Res}(f,g):=
\det
\begin{pmatrix}
A_1 & B_1 & C_1 & 0\\
0 & A_1 & B_1 & C_1\\
A_2 & B_2 & C_2 & 0\\
0 & A_2 & B_2 & C_2
\end{pmatrix}.
$$
The resultant vanishes, $\operatorname{Res}(f,g)=0$, if and only if $f$ and $g$ have a common factor, or equivalently, a common zero in $\mathbb{P}^1$; see, for example, \cite[Section~4.2]{sturmfels2018invitation}.
\end{definition}

The Sylvester matrix in Definition \ref{def:sylvester} packages the condition that two homogeneous binary quadrics share a common root. Indeed, if $f$ and $g$ have a common linear factor $\ell$, then we may write $f=\ell p$ and $g=\ell q$ for linear forms $p$ and $q$. Taking $u=q$ and $v=-p$ gives the nontrivial identity $uf+vg=0$. Writing out the coefficients of this cubic identity in the basis $\lambda_1^3,\lambda_1^2\lambda_2,\lambda_1\lambda_2^2,\lambda_2^3$ gives a homogeneous linear system whose coefficient matrix is exactly the Sylvester matrix above. Hence a nontrivial relation $uf+vg=0$ exists if and only if the Sylvester matrix has nontrivial kernel, equivalently, vanishing determinant. This determinant therefore gives a concrete eliminant for the variables $\lambda_1,\lambda_2$, expressing the existence of a common root as a single polynomial condition on the coefficients of $f$ and $g$. 
\exampleqed

\medskip 

We now apply this criterion to the slice quadrics of \Cref{def:slice-quadrics}. By \Cref{prop:cross-column}, on $\mathcal C_a$ they have the form $q_s=\ell m_s$, where the linear factor $\ell$ is independent of $s$. Restricting a slice quadric to a line in the context-variable space therefore produces a binary quadric, and all such restrictions retain the common linear factor.

Let $L=\PP(\operatorname{Span}\{\xi,\zeta\})\subseteq\PP^{d-1}$, where $\xi,\zeta\in\CC^d$ are linearly independent. Restricting $q_s(y;z)$ to $L$ means substituting $z=\lambda_1\xi+\lambda_2\zeta$, which gives the binary quadric
$$
q_s^L(y;\lambda_1,\lambda_2):=q_s(y;\lambda_1\xi+\lambda_2\zeta).
$$
For some scalars $u_1,\dots,u_d$ and $w_1,\dots,w_d$, we write
$$
Q_u^L:=\sum_{s=1}^d u_s\,q_s^L,
\qquad
Q_w^L:=\sum_{s=1}^d w_s\,q_s^L.
$$

\begin{theorem}[Quartics from Sylvester resultants]\label{thm:slice_resultants}
With notation as above, for every pair of slices $r,s\in[d]$, the Sylvester resultant
$\Res(q_r^L,q_s^L)$
is a quartic polynomial in the coefficient variables $y$ that vanishes on $\mathcal C_a$. More generally,
$\Res(Q_u^L,Q_w^L)$
also vanishes on $\mathcal C_a$. Equivalently, every coefficient of $\Res(Q_u^L,Q_w^L)$ as a polynomial in the auxiliary variables $u_s,w_s$ is a quartic equation for $\mathcal C_a$.
\end{theorem}

\begin{proof}
By \Cref{prop:cross-column}, for every point of $\mathcal C_a$ and every slice index $s$, we have
$$
q_s(z)=\ell(z)m_s(z),
$$
where $\ell(z)$ is independent of $s$. Restricting to the line $L$ gives
$$
q_s^L(\lambda_1,\lambda_2)
=
\ell(\lambda_1\xi+\lambda_2\zeta)\,m_s(\lambda_1\xi+\lambda_2\zeta),
$$
so all restricted binary quadrics share the common factor $\ell(\lambda_1\xi+\lambda_2\zeta)$. Hence $\Res(q_r^L,q_s^L)=0$ on $\mathcal C_a$ for every pair $r,s\in[d]$. The same argument applies to the linear combinations $Q_u^L$ and $Q_w^L$, since both are divisible by the same factor $\ell(\lambda_1\xi+\lambda_2\zeta)$. Therefore $\Res(Q_u^L,Q_w^L)$ also vanishes on $\mathcal C_a$. Finally, each restricted quadric $q_s^L$ has coefficients linear in $y$, and the Sylvester resultant of two binary quadrics is quadratic in the coefficients of each quadric. It follows that $\Res(q_r^L,q_s^L)$ and $\Res(Q_u^L,Q_w^L)$ are homogeneous quartics in $y$.
\end{proof}

\begin{example}\label{ex:resultant_quartic}
Let $d=3$, $d'=1$, $t=2$, and $a\geq1$, and fix $j=1$ and $n=2$. Consider the line
$L=\PP(\operatorname{Span}\{e_1,e_2\})\subseteq\PP^2$, so
$z=\lambda_1e_1+\lambda_2e_2=(\lambda_1,\lambda_2,0)$. 
For each slice $s\in\{1,2,3\}$, the restricted slice quadric has the form
$$
q_s^L(y;\lambda_1,\lambda_2)
:=
q_s(y;\lambda_1e_1+\lambda_2e_2)
=
a_s\lambda_1^2+2b_s\lambda_1\lambda_2+c_s\lambda_2^2,
$$
where
$$
a_s=y_{(1,2),(1,2),(s,1)},\qquad
b_s=y_{(1,2),(2,2),(s,1)},\qquad
c_s=y_{(2,2),(2,2),(s,1)}.
$$
Taking the Sylvester resultant of the first two slices gives the quartic
$$
\operatorname{Res}(q_1^L,q_2^L)
=
a_1^2c_2^2-2a_1a_2c_1c_2-4a_1b_1b_2c_2+4a_1b_2^2c_1
+a_2^2c_1^2+4a_2b_1^2c_2-4a_2b_1b_2c_1.
$$
By \Cref{thm:slice_resultants}, this polynomial vanishes on $\mathcal C_a$.
\end{example}

\noindent \Cref{thm:slice_resultants} applies to every line
$L\subseteq\PP^{d-1}$ in the projective space of the context-column variables. 
In the computations below, we evaluate these resultant constructions only on a finite collection of chosen lines in order to obtain explicit quartic generators.

\begin{example}[\Cref{ex:d3_invariants} revisited]
We now return to the case $d=3$, $d'=1$, $t=2$, and $a=3$ of \Cref{ex:d3_invariants}. There are three slice quadrics $q_1,q_2,q_3$, and hence three restricted binary quadrics $q_1^L,q_2^L,q_3^L$ on any fixed line $L\subseteq\PP^2$. \Cref{thm:slice_resultants} gives two natural sources of quartic invariants on such a line. First, for each unordered pair of slices $\{r,s\}\subseteq\{1,2,3\}$, the Sylvester resultant $\operatorname{Res}(q_r^L,q_s^L)$ vanishes. Since there are three such pairs, this gives three quartics. Second, consider two generic linear combinations
$$
Q_u^L=u_1q_1^L+u_2q_2^L+u_3q_3^L,
\qquad
Q_w^L=w_1q_1^L+w_2q_2^L+w_3q_3^L.
$$
The resultant $\operatorname{Res}(Q_u^L,Q_w^L)$ vanishes identically, so every coefficient in the general variables $u_1,u_2,u_3$, $w_1,w_2,w_3$ is again a quartic invariant. In our computation, we extracted the coefficients of the six mixed monomials
$$
u_1^2w_2w_3,\quad u_2^2w_1w_3,\quad u_3^2w_1w_2,\quad
u_1u_2w_3^2,\quad u_1u_3w_2^2,\quad u_2u_3w_1^2.
$$

For a fixed line, the generic-linear-combination construction produces six quartics, while the pairwise construction produces three quartics.
In our \texttt{Macaulay2} computation, these nine quartics were not
independent: after forming the ideal they generate and applying
\texttt{trim}, we obtained six quartic generators for that line.

To obtain the full family needed to generate $J$, we selected a finite list of lines by hand, computed the resultant ideal associated to each line, and then added these ideals cumulatively. For our chosen collection of lines, the resulting computation recovered the remaining $72$ quartics in the minimal generating set of $J$ from \Cref{ex:d3_invariants}. We do not claim that this choice of lines is canonical or minimal; it is a computational device used to extract explicit generators.
\end{example}

\begin{remark}
In general, it is not clear which finite collections of lines produce enough resultant quartics to recover all generators of this type, nor how such lines should be chosen efficiently. Understanding the dependence on the chosen lines, and developing systematic ways to identify sufficient collections, would be an interesting direction for further study.
\end{remark}

\section{Summary} 
\label{sec:summary}

We have uncovered a rich algebraic structure governing the coefficients of the lightning self-attention map, partitioning these invariants into those arising from single-column monomials and those from cross-column monomials. The roles of these two families depend on the sequence length. Let $Y_{\mathrm{single}}$ and $Y_{\mathrm{cross}}$ denote the sets of formal ambient variables corresponding to the single-column and cross-column coefficients, respectively. In the single-token regime $t=1$, we have $Y_{\mathrm{cross}}=\varnothing$. By \Cref{prop:single-column-structure}, the single-column cubics are exactly the cubics of the form
\[
f=\ell q,
\qquad
\rank(q)\leq 2a.
\]
In particular, they lie on the $(2,1)$ Chow-type variety, with an additional rank constraint inherited from the attention matrix. 
If $a=1$, the single-column variety is the Chow variety of completely reducible~cubic~forms. 

For sequence lengths $t\geq2$, the cross-column coefficients determine the nonlinear geometry. By \Cref{prop:cross-column}, a representative cross-column block is exactly a tuple of quadrics
\[
(q_1,\ldots,q_d)
=
(\ell m_1,\ldots,\ell m_d),
\qquad
\dim\operatorname{Span}\{m_1,\ldots,m_d\}\leq a.
\]
Moreover, the sequence-copy and internal symmetrization relations of \Cref{sec:linear} identify the repeated coefficient blocks and express every variable $y\in Y_{\mathrm{single}}$ as a linear combination of variables in $Y_{\mathrm{cross}}$. Thus, by \Cref{cor:cross-column-reduction}, after imposing the linear relations, the attention variety is linearly isomorphic to the cross-column variety $\mathcal C_a$. In particular, the nonlinear single-column relations are redundant in this regime.

Let $J_{\mathrm{lin}}$ denote the ideal generated by the linear relations of \Cref{sec:linear}, and let
$J \subset \mathbb{R}[Y_{\mathrm{single}} \cup Y_{\mathrm{cross}}]$ be the defining ideal of the attention variety. For $t\geq2$, because $J$ contains $J_{\mathrm{lin}}$, any mixed polynomial
$F(Y_{\mathrm{single}},Y_{\mathrm{cross}})\in J$ reduces via substitution to a polynomial
$\widetilde{F}(Y_{\mathrm{cross}})\in J\cap \mathbb{R}[Y_{\mathrm{cross}}]$.
Thus the study of nonlinear relations may be reduced, modulo the linear relations, to relations among the cross-column coefficients.
Consequently, our search for the defining ideal reduces to classifying the linear relations together with the nonlinear invariants of the corresponding single-column or cross-column variety, depending on the sequence length.

In the rank-one cross-column case, \Cref{thm:rank-one-cross-column}
gives a complete set-theoretic description. In several small-dimensional
examples, our invariant constructions also recover the full defining
ideal computationally. The general set-theoretic and ideal-theoretic
completeness questions remain open and are discussed below.

\subsection*{Invariant families}

Based on the algebraic reductions established above, we isolate the following families of linear and nonlinear invariants for the attention variety.

\begin{itemize}

\item \emph{Linear relations} (\Cref{sec:linear}).
In a fixed output row, the realizable ambient coordinates satisfy sequence-copy relations that identify repeated coefficient blocks as the target column $j$ varies and, for fixed $j$, as the context column $n\neq j$ varies. Furthermore, when $t\geq2$, the internal symmetrization relations express each single-column 
coefficient
$y_j(\mathcal K)$ as a linear combination of cross-column 
coefficients
$y_{n,j}(\mathcal A,b)$ within a fixed output coordinate. 

\item \emph{Single-column Chow-type invariants} (\Cref{sec:single-column}). The single-column coefficients parametrize a cubic polynomial that factors into a linear and a quadratic form, with the quadratic factor having rank at most $2a$, placing the cubic on the $(2,1)$ Chow-type variety. The invariants are obtained from the vanishing maximal 
minors of the Lie flattening matrix $M_{\mathrm{Lie}}(f)$. 
Furthermore, when $2a+1<d$, the low rank of the attention parameter matrix $A=K^\top Q$ bounds the rank of the coefficient submatrix 
$N$, yielding additional determinantal invariants. 
When $a=1$, the single-column variety is the Chow variety of completely reducible cubic forms. 

\item \emph{Cross-column cubic invariants} (\Cref{prop:cross-column-pencil}). For any context column and target column, specific linear combinations of the symmetric tensor slices have rank at most $2$. The vanishing $3 \times 3$ minors of these slice matrices yield a family of cubic invariants. 

\item \emph{Attention rank constraints} (\Cref{prop:cross-column-flattening}). When the attention dimension is bottlenecked ($a < d$), the cross-column flattening matrix $F(y)$ has rank bounded by $a$. The vanishing $(a+1) \times (a+1)$ minors of this matrix generate invariant polynomials of degree $a+1$. 

\item \emph{Rank-one cross-column characterization} (\Cref{thm:rank-one-cross-column}). When $a=1$, the cross-column flattening and rank-two pencil conditions are not only necessary but also sufficient: together, 
the $2\times2$ minors of $F(y)$ and the cubic pencil relations cut out the cross-column variety $\mathcal C_1$ set-theoretically. 

\item \emph{Cross-column determinantal syzygies} (\Cref{thm:context_syzygies}). In the full-rank attention regime ($a \ge d$), dimensional constraints on the target space of the matrix $F(y)$ force exact linear dependencies among its maximal $d \times d$ minors. This dimensional overcrowding generates invariant polynomials of degree $d$.

\item \emph{Veronese-type invariants} (\Cref{thm:veronese}). For any integer $r$ such that $2 \le r \le \min(a,d)$, specific $r \times r$ minors of the cross-column coefficients evaluate to the coordinates of the $r$th Veronese embedding. The $2 \times 2$ minors of the associated catalecticant matrices yield determinantal invariants of degree $2r$. 

\item \emph{Sylvester resultants} (\Cref{thm:slice_resultants}). The quadratic forms defined by the cross-column coefficients share a common linear factor on the attention variety. After restricting to a line, they become binary quadrics, so their Sylvester resultants vanish and yield quartic invariants.

\end{itemize}

\subsection*{Open problems}

While the families of invariants in this paper capture the fundamental geometric constraints of the parametrization, we do not in general claim they form a complete minimal generating set for the defining ideal $J$. An exception is the rank-one cross-column case, where
\Cref{thm:rank-one-cross-column} gives a complete set-theoretic
description.
However, substantially more computational evidence would be needed to formulate a credible conjecture about minimal generators. 
Because symbolic methods quickly become infeasible as $d$, $a$, and $t$ grow, additional generators may appear in larger dimensions. 
Establishing ideal-theoretic completeness would require showing that the ideal generated by our invariants is prime and defines a variety of the same dimension as the attention variety, which appears difficult in general. 
This leaves several directions for future~work. 

\begin{itemize}

\item \emph{Completeness, defining equations, and minimality.} Determine whether the invariant families identified in this paper generate the full defining ideal $J$. Beyond the rank-one cross-column case of \Cref{thm:rank-one-cross-column}, determine which subfamilies already define the attention variety, at least set-theoretically, in different 
regimes of $d$, $a$, $t$. 
More modestly, characterize which invariants remain minimal generators. 

\item \emph{Degree bounds.} Obtain effective bounds on the degrees of minimal generators of the defining ideal in terms of $d$, $a$, and $t$.

\item \emph{Cross-row geometry.} Develop a systematic theory of cross-row invariants for $d'>1$, and study how these relations interact with the rowwise invariant families studied in this paper.

\item \emph{Stabilization under 
architecture growth.} 
As 
$d$, $a$, $t$, and $d'$ increase, the corresponding attention varieties are related by natural embeddings or projections. 
Determine whether the induced chains of defining ideals, under the corresponding maps of coefficient rings, eventually stabilize. 

\item \emph{Beyond the single-head shallow setting.} Derive analogous algebraic invariants for masked self-attention, multi-head attention, and deep attention models. 

\end{itemize}

\subsection*{Acknowledgment} 

This project has been supported by the DARPA AIQ grant HR00112520014. GM was partially supported by 
NSF grants 
DMS-2522495, 
DMS-2145630, 
CCF-2212520, 
the DFG SPP 2298 grant 464109215, 
and the BMFTR in DAAD project 57616814 (SECAI). 

The authors used AI language models ChatGPT (GPT-5.4 Thinking) and Gemini Pro for language editing and exploratory discussions. 
All mathematical results were developed and verified by the authors.

\bibliographystyle{plain} 
\bibliography{references}

@inproceedings{10.1609/aaai.v37i7.26055,
  author       = {Likhosherstov, Valerii and Choromanski, Krzysztof and Weller, Adrian},
  title        = {On the expressive flexibility of self-attention matrices},
  booktitle    = {Proceedings of the Thirty-Seventh AAAI Conference on Artificial Intelligence and Thirty-Fifth Conference on Innovative Applications of Artificial Intelligence and Thirteenth Symposium on Educational Advances in Artificial Intelligence},
  year         = {2023},
  doi          = {10.1609/aaai.v37i7.26055},
}

@inproceedings{NEURIPS2022_1ba5f641,
  author       = {Luo, Shengjie and Li, Shanda and Zheng, Shuxin and Liu, Tie-Yan and Wang, Liwei and He, Di},
  title        = {Your Transformer May Not be as Powerful as You Expect},
  booktitle    = {Advances in Neural Information Processing Systems},
  volume       = {35},
  pages        = {4301--4315},
  year         = {2022},
}

@inproceedings{Yun2020Are,
  author       = {Chulhee Yun and Srinadh Bhojanapalli and Ankit Singh Rawat and Sashank Reddi and Sanjiv Kumar},
  title        = {Are Transformers universal approximators of sequence-to-sequence functions?},
  booktitle    = {International Conference on Learning Representations},
  year         = {2020},
}

@book{landsberg2012tensors,
  author       = {Landsberg, Joseph M.},
  title        = {Tensors: Geometry and Applications},
  series       = {Graduate Studies in Mathematics},
  volume       = {128},
  publisher    = {American Mathematical Society},
  address      = {Providence, RI},
  year         = {2012},
}

@article{pucci1998veronese,
  author       = {Pucci, Mario},
  title        = {The {V}eronese variety and catalecticant matrices},
  journal      = {Journal of Algebra},
  volume       = {202},
  number       = {1},
  pages        = {72--95},
  year         = {1998},
}

@inproceedings{vaswani2017attention,
  author       = {Ashish Vaswani and
               Noam Shazeer and
               Niki Parmar and
               Jakob Uszkoreit and
               Llion Jones and
               Aidan N. Gomez and
               {\L}ukasz Kaiser and
               Illia Polosukhin},
  title        = {Attention Is All You Need},
  booktitle    = {Advances in Neural Information Processing Systems 30},
  pages        = {5998--6008},
  year         = {2017},
}

@article{vuckovic2020mathematical,
  author       = {James Vuckovic and
             Aristide Baratin and
             R{\'e}mi Tachet des Combes},
  title        = {A Mathematical Theory of Attention},
  journal      = {arXiv:2007.02876},
  year         = {2020},
  doi          = {10.48550/arXiv.2007.02876},
}

@article{geshkovski2025mathematical,
  author       = {Borjan Geshkovski and
             Cyril Letrouit and
             Yury Polyanskiy and
             Philippe Rigollet},
  title        = {A Mathematical Perspective on Transformers},
  journal      = {Bulletin of the American Mathematical Society},
  volume       = {62},
  number       = {3},
  pages        = {427--479},
  year         = {2025},
  doi          = {10.1090/bull/1863},
}

@article{lai2024attention,
  author       = {Zehua Lai and
             Lek-Heng Lim and
             Yucong Liu},
  title        = {Attention Is a Smoothed Cubic Spline},
  journal      = {arXiv:2408.09624},
  year         = {2024},
  doi          = {10.48550/arXiv.2408.09624},
}

@inproceedings{henry2025geometry,
  author       = {Nathan W. Henry and
               Giovanni Luca Marchetti and
               Kathl{\'e}n Kohn},
  title        = {Geometry of Lightning Self-Attention: Identifiability and Dimension},
  booktitle    = {The Thirteenth International Conference on Learning Representations},
  year         = {2025},
}

@inproceedings{kim2021lipschitz,
  author       = {Hyunjik Kim and George Papamakarios and Andriy Mnih},
  title        = {The {L}ipschitz Constant of Self-Attention},
  booktitle    = {Proceedings of the 38th International Conference on Machine Learning},
  series       = {Proceedings of Machine Learning Research},
  volume       = {139},
  pages        = {5562--5571},
  year         = {2021},
}

@article{allman2003phylogenetic,
  author       = {Allman, Elizabeth S. and Rhodes, John A.},
  title        = {Phylogenetic invariants for the general {M}arkov model of sequence mutation},
  journal      = {Mathematical Biosciences},
  volume       = {186},
  number       = {2},
  pages        = {113--144},
  year         = {2003},
  doi          = {10.1016/j.mbs.2003.08.004},
}

@article{allman2008phylogenetic,
  author       = {Allman, Elizabeth S. and Rhodes, John A.},
  title        = {Phylogenetic ideals and varieties for the general {M}arkov model},
  journal      = {Advances in Applied Mathematics},
  volume       = {40},
  number       = {2},
  pages        = {127--148},
  year         = {2008},
  doi          = {10.1016/j.aam.2006.10.002},
}

@article{drton2007algebraic,
  author       = {Drton, Mathias and Sullivant, Seth},
  title        = {Algebraic statistical models},
  journal      = {Statistica Sinica},
  volume       = {17},
  pages        = {1273--1297},
  year         = {2007},
}

@inproceedings{kileel2019expressive,
  author       = {Kileel, Joe and Trager, Matthew and Bruna, Joan},
  title        = {On the Expressive Power of Deep Polynomial Neural Networks},
  booktitle    = {Advances in Neural Information Processing Systems 32},
  pages        = {10310--10319},
  year         = {2019},
}

@article{brandenburg2024tropical,
  author       = {Brandenburg, Marie-Charlotte and Loho, Georg and Mont{\'u}far, Guido},
  title        = {The Real Tropical Geometry of Neural Networks for Binary Classification},
  journal      = {Transactions on Machine Learning Research},
  year         = {2024},
}

@article{kubjas2024geometry,
  author       = {Kubjas, Kaie and Li, Jiayi and Wiesmann, Maximilian},
  title        = {Geometry of Polynomial Neural Networks},
  journal      = {Algebraic Statistics},
  volume       = {15},
  number       = {2},
  pages        = {295--328},
  year         = {2024},
  doi          = {10.2140/astat.2024.15.295},
}

@article{alexandr2025constraining,
  author       = {Alexandr, Yulia and Mont{\'u}far, Guido},
  title        = {Constraining the Outputs of {R}e{LU} Neural Networks},
  journal      = {arXiv:2508.03867},
  year         = {2025},
  doi          = {10.48550/arXiv.2508.03867},
}

@article{allman2009identifiability,
  author       = {Allman, Elizabeth S. and Matias, Catherine and Rhodes, John A.},
  title        = {Identifiability of parameters in latent structure models with many observed variables},
  journal      = {The Annals of Statistics},
  volume       = {37},
  number       = {6A},
  pages        = {3099--3132},
  year         = {2009},
  doi          = {10.1214/09-AOS689},
}

@inproceedings{usevich2025identifiability,
  author       = {Usevich, Konstantin and Borsoi, Ricardo and D{\'e}rand, Clara and Clausel, Marianne},
  title        = {Identifiability of Deep Polynomial Neural Networks},
  booktitle    = {Advances in Neural Information Processing Systems},
  year         = {2025},
}

@article{venturi2019spurious,
  author       = {Venturi, Luca and Bandeira, Afonso S. and Bruna, Joan},
  title        = {Spurious Valleys in One-hidden-layer Neural Network Optimization Landscapes},
  journal      = {Journal of Machine Learning Research},
  volume       = {20},
  number       = {133},
  pages        = {1--34},
  year         = {2019},
}

@article{kohn2022geometry,
  author       = {Kohn, Kathl{\'e}n and Merkh, Thomas and Mont{\'u}far, Guido and Trager, Matthew},
  title        = {Geometry of Linear Convolutional Networks},
  journal      = {SIAM Journal on Applied Algebra and Geometry},
  volume       = {6},
  number       = {3},
  pages        = {368--406},
  year         = {2022},
  doi          = {10.1137/21M1441183},
}

@article{alexandr2026verification,
  author       = {Alexandr, Yulia and Duan, Hao and Mont{\'u}far, Guido},
  title        = {Robustness Verification of Polynomial Neural Networks},
  journal      = {arXiv:2602.06105},
  year         = {2026},
  doi          = {10.48550/arXiv.2602.06105},
}

@inproceedings{ehlers2017formal,
  author       = {R{\"u}diger Ehlers},
  title        = {Formal Verification of Piece-Wise Linear Feed-Forward Neural Networks},
  booktitle    = {Automated Technology for Verification and Analysis},
  series       = {Lecture Notes in Computer Science},
  volume       = {10482},
  pages        = {269--286},
  year         = {2017},
  doi          = {10.1007/978-3-319-68167-2_19},
}

@inproceedings{katz2017reluplex,
  author       = {Guy Katz and Clark Barrett and David Dill and Kyle Julian and Mykel Kochenderfer},
  title        = {Reluplex: An Efficient {SMT} Solver for Verifying Deep Neural Networks},
  booktitle    = {Computer Aided Verification},
  series       = {Lecture Notes in Computer Science},
  volume       = {10426},
  pages        = {97--117},
  year         = {2017},
  doi          = {10.1007/978-3-319-63387-9_5},
}

@inproceedings{tjeng2019mip,
  author       = {Vincent Tjeng and Kai Y. Xiao and Russ Tedrake},
  title        = {Evaluating Robustness of Neural Networks with Mixed Integer Programming},
  booktitle    = {International Conference on Learning Representations},
  year         = {2019},
}

@inproceedings{gehr2018ai2,
  author       = {Timon Gehr and Matthew Mirman and Dana Drachsler-Cohen and Petar Tsankov and Swarat Chaudhuri and Martin Vechev},
  title        = {{AI}$^2$: Safety and Robustness Certification of Neural Networks with Abstract Interpretation},
  booktitle    = {2018 IEEE Symposium on Security and Privacy},
  pages        = {3--18},
  year         = {2018},
  doi          = {10.1109/SP.2018.00058},
}

@article{singh2019deeppoly,
  author       = {Gagandeep Singh and Timon Gehr and Markus P{\"u}schel and Martin Vechev},
  title        = {An Abstract Domain for Certifying Neural Networks},
  journal      = {Proceedings of the ACM on Programming Languages},
  volume       = {3},
  number       = {POPL},
  pages        = {1--30},
  year         = {2019},
  doi          = {10.1145/3290354},
  note         = {Article 41},
}

@inproceedings{dvijotham2018dual,
  author       = {Krishnamurthy Dvijotham and Robert Stanforth and Sven Gowal and Timothy A. Mann and Pushmeet Kohli},
  title        = {A Dual Approach to Scalable Verification of Deep Networks},
  booktitle    = {Uncertainty in Artificial Intelligence},
  pages        = {550--559},
  year         = {2018},
}

@inproceedings{wong2018convex,
  author       = {Eric Wong and J. Zico Kolter},
  title        = {Provable Defenses against Adversarial Examples via the Convex Outer Adversarial Polytope},
  booktitle    = {Proceedings of the 35th International Conference on Machine Learning},
  series       = {Proceedings of Machine Learning Research},
  volume       = {80},
  pages        = {5286--5295},
  year         = {2018},
}

@article{bunel2020bab,
  author       = {Rudy Bunel and Jingyue Lu and Ilker Turkaslan and Philip H. S. Torr and Pushmeet Kohli and M. Pawan Kumar},
  title        = {Branch and Bound for Piecewise Linear Neural Network Verification},
  journal      = {Journal of Machine Learning Research},
  volume       = {21},
  number       = {42},
  pages        = {1--39},
  year         = {2020},
}

@inproceedings{wang2021beta,
  author       = {Shiqi Wang and Huan Zhang and Kaidi Xu and Xue Lin and Suman Jana and Cho-Jui Hsieh and J. Zico Kolter},
  title        = {{$\beta$}-{CROWN}: Efficient Bound Propagation with Per-neuron Split Constraints for Neural Network Robustness Verification},
  booktitle    = {Advances in Neural Information Processing Systems 34},
  year         = {2021},
}

@article{albarghouthi2021intro,
  author       = {Aws Albarghouthi},
  title        = {Introduction to Neural Network Verification},
  journal      = {Foundations and Trends in Programming Languages},
  volume       = {7},
  number       = {1--2},
  pages        = {1--157},
  year         = {2021},
  doi          = {10.1561/2500000051},
}

@article{guan2018brill,
  author       = {Yonghui Guan},
  title        = {Brill's equations as a {GL(V)}-module},
  journal      = {Linear Algebra and its Applications},
  volume       = {548},
  pages        = {273--292},
  year         = {2018},
  doi          = {10.1016/j.laa.2018.02.026},
}

@inproceedings{briand2004brill,
  author       = {Emmanuel Briand},
  title        = {Brill's Equations for the Subvariety of Factorizable Forms},
  booktitle    = {Actas del IX Encuentros de {\'A}lgebra Computacional y Aplicaciones (EACA 2004)},
  pages        = {59--63},
  year         = {2004},
}

@misc{M2,
  author       = {Grayson, Daniel R. and Stillman, Michael E.},
  title        = {Macaulay2, a software system for research in algebraic geometry},
  howpublished = {Available at \url{https://www.macaulay2.com}},
}

@book{sturmfels2018invitation,
  author       = {Micha{\l}ek, Mateusz and Sturmfels, Bernd},
  title        = {Invitation to Nonlinear Algebra},
  series       = {Graduate Studies in Mathematics},
  volume       = {211},
  publisher    = {American Mathematical Society},
  address      = {Providence, RI},
  year         = {2021},
}

@article{LORS2019,
  author       = {Samuel Lundqvist and Alessandro Oneto and Bruce Reznick and Boris Shapiro},
  title        = {On generic and maximal $k$-ranks of binary forms},
  journal      = {Journal of Pure and Applied Algebra},
  volume       = {223},
  number       = {5},
  pages        = {2062--2079},
  year         = {2019},
  doi          = {10.1016/j.jpaa.2018.08.015},
}

@inproceedings{fu2023what,
  author       = {Hengyu Fu and Tianyu Guo and Yu Bai and Song Mei},
  title        = {What can a Single Attention Layer Learn? A Study Through the Random Features Lens},
  booktitle    = {Thirty-seventh Conference on Neural Information Processing Systems},
  year         = {2023},
}

@article{cavender1987invariants,
  author       = {Cavender, James A. and Felsenstein, Joseph},
  title        = {Invariants of phylogenies in a simple case with discrete states},
  journal      = {Journal of Classification},
  volume       = {4},
  number       = {1},
  pages        = {57--71},
  year         = {1987},
  doi          = {10.1007/BF01890075},
}

@article{lake1987phylogenetic,
  author       = {Lake, James A.},
  title        = {A rate-independent technique for analysis of nucleic acid sequences: Evolutionary parsimony},
  journal      = {Molecular Biology and Evolution},
  volume       = {4},
  number       = {2},
  pages        = {167--191},
  year         = {1987},
  doi          = {10.1093/oxfordjournals.molbev.a040433},
}

@article{shahverdi2025algebraic,
  author       = {Shahverdi, Vahid},
  title        = {Algebraic complexity and neurovariety of linear convolutional networks},
  journal      = {Acta Universitatis Sapientiae, Mathematica},
  volume       = {17},
  year         = {2025},
  doi          = {10.1007/s44426-025-00002-2},
  note         = {Article 2},
}

@article{finkel2025activation,
  author       = {Finkel, Bella and Rodriguez, Jose Israel and Wu, Chenxi and Yahl, Thomas},
  title        = {Activation degree thresholds and expressiveness of polynomial neural networks},
  journal      = {Algebraic Statistics},
  volume       = {16},
  number       = {2},
  pages        = {113--130},
  year         = {2025},
  doi          = {10.2140/astat.2025.16.113},
}

@inproceedings{marchetti2025algebra,
  author       = {Marchetti, Giovanni Luca and Shahverdi, Vahid and Mereta, Stefano and Trager, Matthew and Kohn, Kathl{\'e}n},
  title        = {Position: Algebra Unveils Deep Learning -- An Invitation to Neuroalgebraic Geometry},
  booktitle    = {Proceedings of the 42nd International Conference on Machine Learning},
  series       = {Proceedings of Machine Learning Research},
  volume       = {267},
  pages        = {81759--81773},
  year         = {2025},
}

@article{grosdos2025algebraicgeometryrationalneural,
  author       = {Grosdos, Alexandros and Robeva, Elina and Zubkov, Maksym},
  title        = {Algebraic Geometry of Rational Neural Networks},
  journal      = {arXiv:2509.11088},
  year         = {2025},
  doi          = {10.48550/arXiv.2509.11088},
}

@article{seigal2017mixtures,
	title={Mixtures and products in two graphical models},
	author={Anna Seigal and Guido Mont\'ufar},
	journal={Journal of Algebraic Statistics},
	volume={9},
	number ={1},
	doi={10.18409/jas.v9i1.90},
	pages={1--20},
	year={2018},	
}

@article{kohn2024,
author = {Kohn, Kathl{\'e}n and Mont{\'u}far, Guido and Shahverdi, Vahid and Trager, Matthew},
title = {Function Space and Critical Points of Linear Convolutional Networks},
journal = {SIAM Journal on Applied Algebra and Geometry},
volume = {8},
number = {2},
pages = {333-362},
year = {2024},
doi = {10.1137/23M1565504}
}

@inproceedings{Trager2019PureAS,
  title={Pure and Spurious Critical Points: a Geometric Study of Linear Networks},
  author={Matthew Trager and Kathl{\'e}n Kohn and Joan Bruna},
  booktitle={International Conference on Machine Learning},
  year={2020}
}

@article{doi:10.1137/16M1077489,
author = {Mont{\'u}far, Guido and Morton, Jason},
title = {Dimension of Marginals of {K}ronecker Product Models},
journal = {SIAM Journal on Applied Algebra and Geometry},
volume = {1},
number = {1},
pages = {126-151},
year = {2017},
doi = {10.1137/16M1077489},
}

@incollection{cueto2010geometry,
  author    = {Cueto, Maria Angelica and Morton, Jason and Sturmfels, Bernd},
  title     = {Geometry of the Restricted {B}oltzmann Machine},
  booktitle = {Algebraic Methods in Statistics and Probability II},
  series    = {Contemporary Mathematics},
  volume    = {516},
  pages     = {135--153},
  publisher = {American Mathematical Society},
  year      = {2010},
  address = {Providence, RI},
}

@article{GARCIA2005331,
title = {Algebraic geometry of {B}ayesian networks},
journal = {Journal of Symbolic Computation},
volume = {39},
number = {3},
pages = {331-355},
year = {2005},
note = {Special issue on the occasion of MEGA 2003},
issn = {0747-7171},
doi = {10.1016/j.jsc.2004.11.007},
author = {Luis David Garcia and Michael Stillman and Bernd Sturmfels}
}

@InProceedings{pmlr-v80-zhang18i,
  title = 	 {Tropical Geometry of Deep Neural Networks},
  author =       {Zhang, Liwen and Naitzat, Gregory and Lim, Lek-Heng},
  booktitle = 	 {Proceedings of the 35th International Conference on Machine Learning},
  pages = 	 {5824--5832},
  year = 	 {2018},
  editor = 	 {Dy, Jennifer and Krause, Andreas},
  volume = 	 {80},
  series = 	 {Proceedings of Machine Learning Research},
  month = 	 {10--15 Jul},
  publisher =    {PMLR},
  address = {Stockholmsmässan, Stockholm, Sweden},
}

@Article{math6120314,
AUTHOR = {Bernardi, Alessandra and Carlini, Enrico and Catalisano, Maria Virginia and Gimigliano, Alessandro and Oneto, Alessandro},
TITLE = {The Hitchhiker Guide to: Secant Varieties and Tensor Decomposition},
JOURNAL = {Mathematics},
VOLUME = {6},
YEAR = {2018},
NUMBER = {12},
ARTICLE-NUMBER = {314},
ISSN = {2227-7390},
DOI = {10.3390/math6120314}
}


\appendix\crefalias{section}{appendix}

\section{An explicit single-column Lie algebra flattening example}
\label{app:single-column-example}

We give here the explicit computation underlying the single-column Lie
algebra flattening invariants in \Cref{ex:type2-relation}. For
$d=a=3$ and $d'=t=1$, we write out the single-column coefficient
parametrization, construct a linear syzygy among the first partial
derivatives, and pass to the corresponding Lie algebra flattening
$M_{\mathrm{Lie}}(f)$. We then record the \texttt{Macaulay2}
computation showing that maximal minors of this flattening generate the
defining ideal in this case.

\begin{example}[$d=3$, $d'=t = 1$] We illustrate the single-column Lie algebra flattening construction
in the smallest nontrivial case by passing from the explicit coefficient parametrization to a linear syzygy matrix, and then to the Lie algebra flattening. 

\medskip
{\noindent \underline{Parametrization and ambient coordinates.}}
Since $d'=t=1$, we suppress the output and column indices and write $y_{k_1k_2k_3}$ for the scaled coefficient of the monomial type $x_{k_1}x_{k_2}x_{k_3}$, where $k_1\le k_2\le k_3$. In this case, $x=(x_1,x_2,x_3)^\top$, $v=[v_1,v_2,v_3]$, and $A=(a_{ij})\in\mathbb{R}^{3\times 3}$. The parametrization is
\begin{equation*}
\scalebox{0.9}{$
\begin{aligned}
\mu: (A, V) \mapsto \Big(&\underbrace{a_{11}v_1}_{y_{111}},\;\underbrace{a_{22}v_2}_{y_{222}},\;\underbrace{a_{33}v_3}_{y_{333}},\;\underbrace{\frac{1}{3}(a_{12}v_1 + a_{21}v_1 + a_{11}v_2)}_{y_{112}},\;\underbrace{\frac{1}{3}(a_{23}v_2 + a_{32}v_2 + a_{22}v_3)}_{y_{223}},\;\underbrace{\frac{1}{3}(a_{33}v_1 + a_{13}v_3 + a_{31}v_3)}_{y_{133}},\\
    &
    \underbrace{\frac{1}{3}(a_{22}v_1 + a_{12}v_2 + a_{21}v_2)}_{y_{122}},\; \underbrace{\frac{1}{3}(a_{33}v_2 + a_{23}v_3 + a_{32}v_3)}_{y_{233}},\;
    \underbrace{\frac{1}{3}(a_{13}v_1 + a_{31}v_1 + a_{11}v_3)}_{y_{113}},\\
    &
    \underbrace{\frac{1}{6}(a_{23}v_1 + a_{32}v_1 + a_{13}v_2 + a_{31}v_2 + a_{12}v_3 + a_{21}v_3)}_{y_{123}}\Big).
\end{aligned}
$}
\end{equation*}

\medskip
\noindent{\underline{Factorization and a linear syzygy.}} The implicit ideal $J$ of $\mathcal{V} = \overline{\mathrm{im}(\mu)}$ has codimension 2 and is generated by 35 generators of degree 8. Since $t=1$, all variables come from a single column $X$ and hence the lightning self-attention map takes the simplified form $x\mapsto x^\top Ax(vx)$. The resulting polynomial factors into a linear and a quadratic form, a geometric property that its coefficients must reflect. We can write $f(x)=x^\top Ax(vx)=\ell(x)q(x)$ for some linear form $\ell$ and quadratic form $q$. Taking the partial derivatives with respect to each variable yields:
\begin{align*}
    f_{x_1} = \ell_{x_1} q + \ell q_{x_1},\qquad
    f_{x_2} = \ell_{x_2} q + \ell q_{x_2},\qquad
    f_{x_3} = \ell_{x_3} q + \ell q_{x_3}.
\end{align*}
Thus each first partial derivative is expressed in terms of the same two factors $\ell$ and $q$. To begin eliminating the common quadratic term, we compare two of these equations. After relabeling coordinates if necessary, we may assume $\ell_{x_1}\neq 0$; for instance,
\begin{align*}
    \ell_{x_2} f_{x_1} &= \ell_{x_2}\ell_{x_1} q + \ell_{x_2}\ell q_{x_1}, \qquad
    \ell_{x_1} f_{x_2} = \ell_{x_1}\ell_{x_2} q + \ell_{x_1}\ell q_{x_2} . 
\end{align*}
\noindent Subtracting the second equation from the first exactly cancels the $q$ term: 
\begin{equation*}
    \ell_{x_2} f_{x_1} - \ell_{x_1} f_{x_2} = \ell (\ell_{x_2} q_{x_1} - \ell_{x_1} q_{x_2}) . 
\end{equation*}
\noindent Notice that $q$ is a quadratic form, so its partial derivatives $q_{x_1}$ and $q_{x_2}$ are linear forms. Let $L_{12}(x) = \ell_{x_2} q_{x_1} - \ell_{x_1} q_{x_2}$, which is also a linear form. 
Thus, we have the relation: 
\begin{equation*}
    \ell_{x_2} f_{x_1} - \ell_{x_1} f_{x_2} = \ell L_{12}. 
\end{equation*}
\noindent Following the same logic for another pair of partial derivatives, we obtain:
\begin{align*}
    \ell_{x_3} f_{x_1} - \ell_{x_1} f_{x_3} &= \ell L_{13} . 
\end{align*}
\noindent We can then eliminate the common linear factor $\ell$ by cross-multiplying the 
two relations by $L_{13}$ and $L_{12}$ respectively:
\begin{equation*}
    L_{13}(\ell_{x_2} f_{x_1} - \ell_{x_1} f_{x_2}) = L_{13} \ell L_{12} = L_{12}(\ell_{x_3} f_{x_1} - \ell_{x_1} f_{x_3}) . 
\end{equation*}
\noindent Rearranging this equation yields a syzygy among the partial derivatives of $f$:
\begin{equation*}
    (L_{13}\ell_{x_2} - L_{12}\ell_{x_3})f_{x_1} - (L_{13}\ell_{x_1})f_{x_2} + (L_{12}\ell_{x_1})f_{x_3} = 0 . 
\end{equation*}

\noindent Because the $L_{ij}$ are linear forms in $x$ and the $\ell_{x_i}$ are constants, the coefficients in parentheses are entirely linear forms in $x$. This demonstrates that the partial derivatives $f_{x_i}$ are linearly dependent over the ring of polynomials, specifically exhibiting a syzygy with linear coefficients.

\medskip
\noindent{\underline{The linear syzygy matrix.}} Hence, to explicitly compute this relation, we search for the coefficients of three general linear forms $\alpha, \beta, \gamma$ such that $\alpha f_{x_1} + \beta f_{x_2} + \gamma f_{x_3} = 0$. We parameterize these forms as:
\begin{align*}
    \alpha(x) &= \alpha_1 x_1 + \alpha_2 x_2 + \alpha_3 x_3, \\
    \beta(x) &= \beta_1 x_1 + \beta_2 x_2 + \beta_3 x_3, \\
    \gamma(x) &= \gamma_1 x_1 + \gamma_2 x_2 + \gamma_3 x_3.
\end{align*}

Substituting these parameterized expressions back into the equation $\alpha f_{x_1} + \beta f_{x_2} + \gamma f_{x_3} = 0$ multiplies our unknown linear forms by the quadratic forms $f_{x_i}$. Expanding these products yields a single homogeneous cubic polynomial in the variables $x_1, x_2, x_3$. For this relation to hold universally, the resulting cubic polynomial must identically vanish, meaning the aggregated coefficient associated with each of the $10$ standard cubic monomials ($x_1^3, x_1^2x_2, \dots, x_3^3$) must be exactly zero.

By gathering the terms for each of these $10$ monomials, we obtain a system of $10$ linear equations in our $9$ unknown coefficients ($\alpha_i, \beta_i, \gamma_i$). The coefficient matrix of this linear system is precisely the $10 \times 9$ linear syzygy matrix $M$, where the rows correspond to the cubic monomials and the columns represent the unknowns. To make this construction explicit, we introduce variables $a,b,c,\dots,j$ for the scaled coefficient coordinates of the self-attention cubic $f(x)$, namely:
\begin{align*}
    a &= y_{111},\; b = y_{222},\; c = y_{333},\; d = y_{112},\; e = y_{223},\\
    f &= y_{133},\; 
    g = y_{122},\; 
    h = y_{233},\; 
    i = y_{113},\; 
    j = y_{123}.
\end{align*}
\noindent Expanding the equation $\alpha f_{x_1} + \beta f_{x_2} + \gamma f_{x_3} = 0$ and extracting the linear system $M \mathbf{v} = 0$, where $\mathbf{v} = (\alpha_1, \alpha_2, \alpha_3, \beta_1, \beta_2, \beta_3, \gamma_1, \gamma_2, \gamma_3)^\top$, yields the following explicit matrix:

$$
M = \begin{array}{c|ccccccccc}
& \alpha_1 & \alpha_2 & \alpha_3 & \beta_1 & \beta_2 & \beta_3 & \gamma_1 & \gamma_2 & \gamma_3 \\
\hline
x_1^3     & a  & 0 & 0 & d  & 0  & 0 & i  & 0 & 0 \\
x_1^2x_2  & 2d & a & 0 & 2g & d  & 0 & 2j & i & 0 \\
x_1^2x_3  & 2i & 0 & a & 2j & 0  & d & 2f & 0 & i \\
x_1x_2^2  & g  & 2d& 0 & b  & 2g & 0 & e  & 2j& 0 \\
x_1x_2x_3 & 2j & 2i& 2d& 2e & 2j & 2g& 2h & 2f& 2j \\
x_1x_3^2  & f  & 0 & 2i& h  & 0  & 2j& c  & 0 & 2f \\
x_2^3     & 0  & g & 0 & 0  & b  & 0 & 0  & e & 0 \\
x_2^2x_3  & 0  & 2j& g & 0  & 2e & b & 0  & 2h& e \\
x_2x_3^2  & 0  & f & 2j& 0  & h  & 2e& 0  & c & 2h \\
x_3^3     & 0  & 0 & f & 0  & 0  & h & 0  & 0 & c
\end{array}
$$
For a general factorization $f=\ell q$, the construction above gives a non-trivial linear syzygy among its partial derivatives. This guarantees that there exists at least one set of linear forms $\alpha, \beta, \gamma$ (not all identically zero) such that $\alpha f_{x_1} + \beta f_{x_2} + \gamma f_{x_3} = 0$. Translating this back to our matrix formulation, this guarantee means there is a non-zero solution vector $\mathbf{v}$ that satisfies the homogeneous linear system $M \mathbf{v} = 0$. Hence, the kernel of $M$ is non-trivial. Consequently, the columns of $M$ must be linearly dependent. Since $M$ is a $10 \times 9$ matrix, its maximum possible column rank is $9$. However, because its columns are linearly dependent, all ten maximal $9 \times 9$ minors must vanish. 

\medskip
\noindent  {\underline{Passing from $\mathfrak{gl}_3$ to $\mathfrak{sl}_3$.}} However, the matrix $M$ represents the infinitesimal $\mathfrak{gl}_3$-action on cubic forms via first-order differential operators. Because $f(x)$ is homogeneous of degree~$3$, the scalar direction acts by the Euler operator. Applying the Euler operator $E = x_1 \partial_{x_1} + x_2 \partial_{x_2} + x_3 \partial_{x_3}$ simply scales the polynomial: $E(f) = 3f$. This corresponds to the scalar matrix $I \in \mathfrak{gl}_3$. In terms of our parametrization, the trace is $\alpha_1+\beta_2+\gamma_3$, which is the sum of the diagonal coefficients of the associated linear vector field. To isolate the non-trivial geometric invariants of the self-attention map, we remove this one-dimensional scaling symmetry and restrict our attention to the $8$-dimensional special linear Lie algebra $\mathfrak{sl}_3$, consisting of trace-zero operators.

\medskip
\noindent{\underline{The Lie algebra flattening matrix.}} We can systematically transform our $10 \times 9$ matrix $M$ into a $10 \times 8$ matrix by changing our basis to the $8$ standard generators of $\mathfrak{sl}_3$. We retain the $6$ off-diagonal elements ($E_{ij} = x_i \partial_{x_j}$ for $i \neq j$). We replace the $3$ diagonal elements with the $2$ trace-zero diagonal generators $H_1 = x_1 \partial_{x_1} - x_3 \partial_{x_3}$ and $H_2 = x_2 \partial_{x_2} - x_3 \partial_{x_3}$. Changing basis and restricting to the trace-zero subspace yields the $10 \times 8$ matrix $M_{\mathrm{Lie}}$ (whose columns correspond to $H_1, H_2, E_{12}, E_{13}, E_{21}, E_{23}, E_{31}, E_{32}$): 
$$
M_{\mathrm{Lie}} = \begin{array}{c|cccccccc}
& H_1 & H_2 & E_{12} & E_{13} & E_{21} & E_{23} & E_{31} & E_{32} \\
\hline
x_1^3     & a  & 0  & d  & i  & 0  & 0  & 0  & 0 \\
x_1^2x_2  & 2d & d  & 2g & 2j & a  & i  & 0  & 0 \\
x_1^2x_3  & i  & -i & 2j & 2f & 0  & 0  & a  & d \\
x_1x_2^2  & g  & 2g & b  & e  & 2d & 2j & 0  & 0 \\
x_1x_2x_3 & 0  & 0  & 2e & 2h & 2i & 2f & 2d & 2g \\
x_1x_3^2  & -f & -2f& h  & c  & 0  & 0  & 2i & 2j \\
x_2^3     & 0  & b  & 0  & 0  & g  & e  & 0  & 0 \\
x_2^2x_3  & -e & e  & 0  & 0  & 2j & 2h & g  & b \\
x_2x_3^2  & -2h& -h & 0  & 0  & f  & c  & 2j & 2e \\
x_3^3     & -c & -c & 0  & 0  & 0  & 0  & f  & h
\end{array}
$$

\medskip
\noindent {\underline{Trace-zero syzygy and vanishing minors.}} The non-trivial linear syzygy we established earlier survives this restriction to the trace-zero subspace. To see this, we can compute the trace of the corresponding differential operator $D = \alpha \partial_{x_1} + \beta \partial_{x_2} + \gamma \partial_{x_3}$. The trace of $D$ as an element of $\mathfrak{gl}_3$ is given by its divergence: $\operatorname{Tr}(D) = \frac{\partial \alpha}{\partial x_1} + \frac{\partial \beta}{\partial x_2} + \frac{\partial \gamma}{\partial x_3}$. Recall our explicit coefficients for the syzygy:
\begin{align*}
    \alpha &= L_{13}\ell_{x_2} - L_{12}\ell_{x_3}, \quad\quad\quad   
    \beta = -L_{13}\ell_{x_1}, \quad\quad\quad  
    \gamma = L_{12}\ell_{x_1}
\end{align*}

\noindent Expanding $\alpha$ using the definitions $L_{12} = \ell_{x_2}q_{x_1} - \ell_{x_1}q_{x_2}$ and $L_{13} = \ell_{x_3}q_{x_1} - \ell_{x_1}q_{x_3}$, we find:
\begin{equation*}
    \alpha = (\ell_{x_3}q_{x_1} - \ell_{x_1}q_{x_3})\ell_{x_2} - (\ell_{x_2}q_{x_1} - \ell_{x_1}q_{x_2})\ell_{x_3} = \ell_{x_1}(\ell_{x_3}q_{x_2} - \ell_{x_2}q_{x_3}) . 
\end{equation*}

\noindent Applying the same expansion to $\beta$ and $\gamma$ yields a highly symmetric set of coefficients: 
\begin{align*}
    \alpha &= \ell_{x_1}(\ell_{x_3}q_{x_2} - \ell_{x_2}q_{x_3}), \quad \quad 
    \beta = \ell_{x_1}(\ell_{x_1}q_{x_3} - \ell_{x_3}q_{x_1}), \quad \quad 
    \gamma = \ell_{x_1}(\ell_{x_2}q_{x_1} - \ell_{x_1}q_{x_2})
\end{align*}

\noindent Because $\ell$ is linear and $q$ is quadratic, the derivatives $\ell_{x_i}$ and $q_{x_i x_j}$ are constants. Taking the partial derivatives of these coefficients gives:

\[
\scalebox{1}{$
\begin{aligned}
\frac{\partial \alpha}{\partial x_1}
&= \ell_{x_1}(\ell_{x_3}q_{x_1 x_2} - \ell_{x_2}q_{x_1 x_3}),
\quad
\frac{\partial \beta}{\partial x_2}
= \ell_{x_1}(\ell_{x_1}q_{x_2 x_3} - \ell_{x_3}q_{x_1 x_2}),
\quad
\frac{\partial \gamma}{\partial x_3}
= \ell_{x_1}(\ell_{x_2}q_{x_1 x_3} - \ell_{x_1}q_{x_2 x_3}).
\end{aligned}
$}
\]

\noindent Summing these three terms, all mixed partial derivatives cancel out,~giving~$\operatorname{Tr}(D) = 0$. Because the constructed syzygy has trace zero, the operator $D$ naturally resides in the $8$-dimensional special linear Lie algebra $\mathfrak{sl}_3$. Therefore, quotienting out the $1$-dimensional scaling symmetry does not annihilate our solution vector, and the matrix $M_{\mathrm{Lie}}$ must drop rank. So, $M_{\mathrm{Lie}}$ is a $10\times 8$ matrix for which all $\binom{10}{8} = 45$ maximal $8\times 8$ minors must vanish. 

A \texttt{Macaulay2} \cite{M2} computation shows that $J$ is fully generated by 35 degree-8 such minors; the remaining maximal minors lie in the ideal they generate. Thus, in this example, the Lie algebra flattening invariants generate
the full defining ideal $J$ and hence characterize the attention variety
ideal-theoretically. The cross-column invariant families do not arise
when $t=1$, while \Cref{prop:rank-bounds} gives no additional
determinantal constraints since $2a+1=7>d=3$.
\end{example}

\section{Cross-row relations}
\label{app:cross-row}

In this article we mainly study \emph{rowwise} invariants, obtained after reducing to the case $d'=1$. 
For general $d'$, however, the attention variety also satisfies polynomial relations coupling different output rows. 
The reason is that all output rows depend on the same attention matrix $A=K^\top Q$, while the dependence on the value matrix $V$ is linear in the chosen row. 

This is already visible in the coefficient formulas of \Cref{prop:sums-monomials}. For each fixed monomial label, the corresponding scaled coefficient is a linear form in the entries of the relevant row of $V$, with coefficients determined by $A$ alone. As a result, any collection of coefficients taken across output rows factors through the common matrix~$V$.

\begin{proposition}[Cross-row determinantal relations]
Fix an output column $j\in[t]$, and let $S$ be any finite collection of ambient coordinates among the variables $y^{[i]}_j(\mathcal K)$ and $y^{[i]}_{n,j}(\mathcal A,b)$. Let $C_S(W)\in \mathbb R^{d'\times |S|}$ be the matrix whose $(i,s)$-entry is the scaled coefficient indexed by the coordinate $s\in S$ in output row $i$. Then there exists a matrix $M_S(A)\in \mathbb R^{d\times |S|}$, depending only on the attention matrix $A$, such that $C_S(W)=V\,M_S(A)$. Consequently, $\rank C_S(W)\le d$, and all $(d+1)\times(d+1)$ minors of $C_S(W)$ vanish on the attention variety.
\end{proposition}

\begin{proof}
By \Cref{prop:sums-monomials}, every unscaled coefficient is a sum of terms of the form $a_{\bullet\bullet}v_{ip}$. After dividing by the appropriate orbit size, the same remains true for the scaled coefficients. Thus, for each chosen coordinate $s\in S$, the corresponding entry in output row $i$ has the form
\[
\sum_{p=1}^d v_{ip}\,m_{p,s}(A),
\]
where the coefficients $m_{p,s}(A)$ depend only on $A$. Collecting these coefficients over all $s\in S$ gives a matrix $M_S(A)\in\mathbb R^{d\times |S|}$ such that
$
C_S(W)=V\,M_S(A).
$
The rank bound and the vanishing of all $(d+1)\times(d+1)$ minors follow immediately.
\end{proof}

These equations give a basic family of cross-row invariants. We record them here for completeness; a systematic study of their interaction with the rowwise invariant families developed in the main body of the paper is left for future~work.

\end{document}